\documentclass[12pt]{amsart}

\usepackage[all,color]{xy}
\usepackage{pb-diagram}
\usepackage[mathscr]{eucal}
\usepackage{hyperref}
\hypersetup{
    colorlinks=true, %set true if you want colored links
    linktoc=all,     %set to all if you want both sections and subsections linked
    linkcolor=blue,  %choose some color if you want links to stand out
}

\usepackage{xcolor}
\usepackage[normalem]{ulem}

\usepackage[toc,page]{appendix}
\usepackage{pifont}
\usepackage{combelow} % typeset comma-below letters as in the Romanian language

\RequirePackage{ifthen,setspace,enumitem,tikz}
\usetikzlibrary{arrows}

\DeclareMathAlphabet{\mathpzc}{OT1}{pzc}{m}{it}

\usepackage{amsfonts}
\usepackage{amsmath}
\usepackage{amssymb}

\newcommand{\ric}{\textnormal{Ric}}

\newcommand{\dbar}{\overline{\partial}}

\newcommand{\R}[4]{R_{#1 \overline{#2} #3 \overline{#4}}}

\newcommand{\diam}{\textnormal{Diam}}
\newcommand{\RCD}{\textnormal{RCD}}

\newcommand{\ddbar}{\sqrt{-1}\partial\dbar}

\DeclareMathOperator{\reg}{reg}
\DeclareMathOperator{\red}{red}
\DeclareMathOperator{\Exc}{Exc}
\DeclareMathOperator{\NNef}{NNef}
\DeclareMathOperator{\Ram}{Ram}

\def\cK{{\mathcal K}}
\def\cL{{\mathcal L}}

\def\cS{{\mathcal S}}

\def\Z{{\mathbb Z}}
\def\Q{{\mathbb Q}}
\def\R{{\mathbb R}}
\def\C{{\mathbb C}}

\def\B{{\mathbb B}} % stable base loci
\def\PSH{\textnormal{PSH}}
\def\Vol{\textnormal{Vol}}
\def\nilp{\textnormal{nilp}}
\def\Supp{\textnormal{Supp}}
\def\BottChern{\textnormal{BC}}
\def\Sym{\textnormal{Sym}}

\newcommand{\Alb}[0]{\operatorname{Alb}}
\newcommand{\alb}[0]{\operatorname{alb}}

\newtheorem{theorem}{Theorem}[section]
\newtheorem{proposition}{Proposition}[section]
\newtheorem{lemma}{Lemma}[section]
\newtheorem{definition}{Definition}[section]
\newtheorem{corollary}{Corollary}[section]
\newtheorem{remark}{Remark}[section]
\newtheorem{conjecture}{Conjecture}[section]

\numberwithin{equation}{section}

\begin{document}

\title[Fundamental groups of compact K{\"a}hler varieties]{Fundamental groups of compact K{\"a}hler varieties with nef anticanonical bundle}

\author[{Xin Fu, Bin Guo, Jian Song and Juanyong Wang}]{Xin Fu $^*$, Bin Guo$^\dagger$ , Jian Song$^{\dagger \dagger}$, Juanyong Wang $^{**}$ }

\thanks{Xin Fu is  supported by National Key R\&D Program of China 2024YFA1014800 and NSFC No. 12401073. Bin Guo and Jian Song are supported in part by the National Science Foundation under grants DMS-2203607 and DMS-2303508, and the collaboration grant 946730 from Simons Foundation. Juanyong Wang is supported by NSFC (Grants No.12301060 and 12288201) and National Key R\&D Program of China (Grant No. 2021YFA1003100).}

\address{$^*$ School of Science, Institute for Theoretical Sciences, Westlake University, Hangzhou 310030, China}

\email{fuxin54@westlake.edu.cn}

\address{$^\dagger$ Department of Mathematics \& Computer Science, Rutgers University, Newark, NJ 07102}

\email{bguo@rutgers.edu}

\address{$^{\dagger\dagger}$ Department of Mathematics, Rutgers University, Piscataway, NJ 08854}

\email{jiansong@math.rutgers.edu}

\address{$^{**}$ State Key Laboratory of Mathematical Sciences, Academy of Mathematics and Systems Science, Chinese Academy of Sciences, Beijing 100190, China}

\email{juanyong.wang@amss.ac.cn}

\begin{abstract}  It is proved by M. P{\u a}un (1997, 2017) that the fundamental group of a compact K{\"a}hler manifold $X$ is almost Abelian if the anti-canonical bundle $-K_X$ is nef. In this paper, we apply the recent geometric analytic theory of K{\"a}hler spaces developed by Guo-Phong-Song-Sturm to study fundamental groups of mildly singular compact K{\"a}hler varieties. We first extend P{\u a}un's result to log canonical pairs $(X,\Delta)$ with smooth $X$ and nef $-(K_X+\Delta)$ as well as to compact K{\"a}hler manifolds $X$ with pseudo-effective $-K_X$ under a suitable assumption on the singularities of $c_1(-K_X)$. We further prove that, for a $3$-dimensional  log canonical pair $(X,\Delta)$ with $X$ being klt, $\pi_1(X)$ is almost Abelian if $-(K_X+\Delta)$ is nef. Moreover, as one of the main ingredients for the proof of these results, we establish the surjectivity of the Albanese maps of compact normal complex varieties $X$ in Fujiki class $\mathscr{C}$ that admits an effective $\R$-divisor $\Delta$ such that the pair $(X,\Delta)$ is log canonical with nef anti-log canonical divisor $-(K_X+\Delta)$. This generalizes the corresponding theorems for projective varieties (Zhang, 2005), for klt pairs (Matsumura-Wang-Wu-Zhang, 2025) and for log smooth case (Fu-Han-Zou, 2025).
{\footnotesize }

\end{abstract}

\maketitle

% {\footnotesize \tableofcontents}

\section{Introduction}

Throughout this paper, we work over complex number $\C$. By abuse of notations, we adopt the additive notation for tensor products of line bundles (e.g., $L + M := L \otimes M$ for line bundles $L$ and $M$), and interchangeably use it with the addition of divisors; moreover, for a pair $(X,\Delta)$ with $X$ being a complex variety and $\Delta$ an $\R$-divisor such that $K_X+\Delta$ is an $\R$-line bundle, by abuse of notations, we call $K_X+\Delta$ the log canonical divisor of $(X,\Delta)$, although it is not necessarily an $\R$-divisor. A \textit{fibration} is a proper surjective morphism with connected fibers. 

\subsection{Motivation}
\label{ss:motivation}

Mildly singular compact K{\"a}hler varieties with nef anticanonical divisor, which appear as natural generalizations of Calabi-Yau manifolds and Fano varieties, especially in the framework of generalized pairs developed by Birkar-Zhang \cite{BZ16}, have been extensively studied over the last 30 years. Although the structure of the canonical fibrations associated with these varieties (namely, the Albanese maps and the MRC fibrations) have been elucidated by successive works \cite{DPS93,Zha96,Zha05,Pau01,Pau17,Cao19,CH19,CCM19,Wan22,MW,MWWZ25}, their fundamental groups are still far from well understood. Indeed, we have the following conjecture:
\begin{conjecture}
\label{conj:main}
Let $X$ be a compact normal K{\"a}hler variety, assume that there is an effective divisor $\Delta$ on $X$ such that $(X,\Delta)$ is a log canonical (lc) pair and the anti-log canonical divisor $-(K_X+\Delta)$ is nef. Then the fundamental group $\pi_1(X_{\reg})$ is almost Abelian.
\end{conjecture}

If $X$ is smooth and $\Delta=0$, the conjecture has been thoroughly settled, essentially due to M. P{\u a}un \cite{Pau97} (combined with the surjectivity result for the Albanese maps \cite{Pau17}). The result of P{\u a}un is further generalized to the orbifold case by \cite{Liu23}. When $X$ is singular or $\Delta\neq0$, the conjecture is still widely open except in the following two `extreme' cases:
\begin{itemize}
\item If $(X,\Delta)$ is weakly Fano (i.e. $-(K_X+\Delta)$ is big and nef), the conjecture can be reduced to the klt case and was known as the Conjecture of Gurjar-Zhang (see \cite{GZ94,GZ95}). In this case, a weak form of the conjecture (stated for the \'etale fundamental groups) was shown by \cite{Xu14,GKP16b}, then the conjecture was completely solved by L. Braun \cite{Bra21} in this case, based on the idea of Tian-Xu \cite{TX17}.
\item If $\Delta=0$, $X$ is Calabi-Yau (i.e. $K_X\equiv 0$), then it is shown by Greb-Guenancia-Kebekus \cite[Corollary 13.10]{GGK19} that any finite-dimensional representation of $\pi_1(X_{\reg})$ has almost Abelian image. See \cite[Theorem 13.1]{GGK19} and \cite[Proposition 8.23]{GKP16a} for relevant results. 
\end{itemize}

Instead of considering the fundamental groups of the regular loci $X_{\reg}$\,, which is quite difficult to study due to the lack of tools from both algebraic geometry and complex differential geometry, we consider in this paper the fundamental group $\pi_1(X)$ of the variety $X$ itself. Similar to Conjecture \ref{conj:main}, we also have the almost abelianity conjecture for $\pi_1(X)$. 

\begin{conjecture}
\label{conj:main1}
Let $X$ be a compact normal K{\"a}hler variety, assume that there is an effective divisor $\Delta$ on $X$ such that $(X,\Delta)$ is a lc pair and the anti-log canonical divisor $-(K_X+\Delta)$ is nef. Then the fundamental group $\pi_1(X)$ is almost Abelian.
\end{conjecture}

Note that by \cite[\S 0.7(B), p.33]{FL81} we have a surjection $\pi_1(X_{\reg})\twoheadrightarrow\pi_1(X)$, which means that $\pi_1(X_{\reg})$ contains more information than $\pi_1(X)$ although the former is usually much easier to tackle than the latter; indeed, for a singular variety $X$, in order to understand its topology, we need not only know the information of its \'etale covers (encoded by $\pi_1(X)$), but also that of its quasi-\'etale covers (encoded by $\pi_1(X_{\reg})$). In particular, we see that Conjecture \ref{conj:main1} is indeed a corollary of Conjecture \ref{conj:main}.
%it has the following two disadvantages: 
%\begin{itemize}
%\item First, for a singular variety $X$, we need not only know the information of its \'etale covers, which is encoded by $\pi_1(X)$, but also that of its quasi-\'etale covers, which is encoded by $\pi_1(X_{\reg})$;  
%\item Second, by \cite[Theorem 1.5]{GKP16b} it suffices to prove Conjecture \ref{conj:main} for maximally quasi-\'etale varieties (see \cite[Definition 2.11]{MW}), and in view of the structure theorem established by \cite{MW}, we can further reduce the conjecture to the case of rationally connected varieties and the case of Calabi-Yau varieties; however, the fundamental group $\pi_1(X)$ does not behave well under the maximally quasi-\'etale cover (see \cite[\S 1.15]{GZ94}), thus we cannot reduce Conjecture \ref{conj:main1} to the Calabi-Yau case.   
%\end{itemize}

Moreover, in the klt case, \cite[Theorem 1.5]{GKP16b} and \cite[Proposition 1.3]{Cam91} enable us to reduce Conjectures \ref{conj:main} and \ref{conj:main1} to the case of maximally quasi-\'etale varieties (see \cite[Definition 2.11]{MW}); and in view of the structure theorem established by \cite{MW}, if $X$ is projective we can further reduce Conjecture \ref{conj:main} (resp. Conjecture \ref{conj:main1}) to the case of rationally connected varieties and the case of Calabi-Yau varieties (resp. to the case of Calabi-Yau varieties, since rationally connected varieties are simply connected by \cite[Theorem 3.5]{Cam91}). The lc case remains much difficult for both conjectures.

Finally let us point out the relation of Conjectures \ref{conj:main} and \ref{conj:main1} to the famous conjectures of Koll{\'a}r (\cite{Kol95}) and Campana (\cite{Cam04}). 

First, it is clear that under the assumption of the existence of good minimal models, Conjecture \ref{conj:main1} implies the following conjecture of Koll{\'a}r:
\begin{conjecture}[{\cite[4.16, p.54]{Kol95}}] 
\label{Kollar} 
Let $X$ be an $n$-dimensional projective manifold with $\kappa(X)=0$, then $\pi_1(X)$ is almost abelian.
\end{conjecture}

Koll{\'a}r's conjecture is shown in complex dimension 3 \cite{Kol95} and 4 \cite{GKP16a} and is largely open in higher dimensions. 

Furthermore, both Conjecture \ref{conj:main1} and Conjecture \ref{Kollar} are special cases of Campana's Abelianity conjecture for special varieties \cite[Conjecture 7.1]{Cam04}. One should notice that the orbifold version of the Abelianity conjecture \cite[Conjecture 7.3]{Cam04} is not true for lc pais, as illustrated by \cite[Example 10.2]{CLM23}.
%while Conjecture \ref{conj:main} can be regarded a special case of the orbifold version of the Abelianity conjecture \cite[Conjecture 7.3]{Cam04} if $\Delta$ is a $\Q$-divisor.
Indeed, we will show in \S \ref{ss:special} that the variety $X$ (resp. the pair $(X,\Delta)$) in the statement of Conjecture \ref{conj:main} is special in the sense of \cite[Definition 2.1(2)]{Cam04} (resp. a special orbifold in the sense of \cite[Definition 2.41]{Cam04} if $\Delta$ is a $\Q$-divisor). If $X$ is projective, these specialness results have already been established by F. Campana in \cite[Theorem 10.3]{Cam16}. Moreover, the linear version of the Abelianity conjecture \cite[Conjecture 7.1]{Cam04} is known by \cite[Theorem 7.8]{Cam04} (see also \cite[Proposition 3.9]{LOWYZ24}).

\subsection{Main results} 
The recent development in K\"ahler metric measure spaces \cite{CCT02,Song14I,Song14II,SSW20,Sze25, Sze25II,GPSS1,GPSS2,GPSS3,GS,FGS25,CCHSTT25} provides a crucial link between complex geometry and differential geometry via nonlinear PDEs. In particular, in this paper, the RCD structure on such K\"ahler spaces  \cite{Sze25,FGS25} enables applications of Riemannian geometry to the geometric and topological structures of local and global K\"ahler spaces with singularities. A main question in this exciting new research field is the following conjecture. 
\begin{conjecture}\label{conrcd}
    Let $X$ be a normal K\"ahler variety with klt singularities. Suppose \begin{enumerate}
        \item $\omega$ is a Kahler current with $L^p$ volume measure for some $p>1$,
        \item the Ricci current of $\omega$ is bounded below by $-\omega$. 
        \item $\omega$ is smooth on a Zariski open subset $X^\circ$ of $X$.
    \end{enumerate} 
    Then the metric completion of the metric measure space $(X^\circ, \omega, \omega^n)$ is a non-collapsed RCD space homeomorphic to $X$.
    
    \end{conjecture}
Conjecture \ref{conrcd} is confirmed in the case of $3$-dimensional projective varieties \cite{FGS25} and in the case of K\"ahler spaces whose singularities can be resolved with relative nef anti-canonical bundles \cite{Sze25} or relative effective anti-canonical bundles \cite{FGS25}.

Our first main result generalizes P{\u a}un's theorem to lc pairs $(X,\Delta)$, where $\Delta$ always denotes an effective $\mathbb R$-divisor (not necessarily with simple normal crossing support) unless otherwise mentioned. More precisely, we prove the following theorem:

%\textcolor{red}{with $X$ being smooth and $\Delta$ not necessarily SNC.Let us remark that when $X$ is smooth, we expect Conjectures \ref{conj:main} and \ref{conj:main1} to hold for log canonical pairs (in the singular case we feel that this expectation might be a bit too optimistic). }

\begin{theorem} \label{thm:main0} 
Let $X$ be a compact klt K\"ahler variety paired with an effective $\R$-divisor $\Delta$ such that $(X,\Delta)$ is log canonical and $-(K_X+\Delta)$ is nef. Then the fundamental group $\pi_1(X)$ is almost Abelian if Conjecture \ref{conrcd} holds. 

\end{theorem}

Theorem \ref{thm:main0} is a showcase that displayes the  connections among complex PDEs, Riemannian geometry, algebraic geometry and topology . We will prove that Conjecture \ref{conrcd} always holds if $X$ is smooth or $\dim_\C X =3$, which immediately leads to the following corollary of 
Corollary \ref{cor:main1}.

\begin{corollary}\label{cor:main1} Let $X$ be a compact klt K\"ahler variety paired with an effective $\R$-divisor $\Delta$ such that $(X,\Delta)$ is log canonical and $-(K_X+\Delta)$ is nef. Then the fundamental group $\pi_1(X)$ is almost Abelian
in  either of the following two cases:
\begin{enumerate}

\item $X$ is smooth, 

\item $\dim_{\C} X=3$. 

\end{enumerate}

\end{corollary}

We remark that  Conjecture \ref{conj:main1} can be established for higher dimensional singular varieties in some special cases where certain special resolution of singularities for $X$ are assumed, see Theorem \ref{thm:main2}.

%Concerning the item (1), when $\Delta$ is SNC with coefficients of the form $1-\frac{1}{a}, a\in \mathbb{Z}_{>0}$, it is obtained by Zhining Liu in \cite{Liu23}. Moreover, 
In the statement of Corollary \ref{cor:main1}, if we assume the pair $(X,\Delta)$ to be klt,  case (1) can be obtained from the structure theorem \cite[Theorem 1.2]{MWWZ25} and the classical Beauville-Bogomolov decomposition \cite{Bea83}; and if we further assume $X$ to be projective, case (2) can be obtained by \cite[Theorem 1.2]{MW} and \cite[Theorem 6]{BF24} . However, in order to establish the log canonical case, for which the structure theorem of the form \cite{MWWZ25} fails (see \cite{BFPT24}), much extra effort is needed. Indeed, even the log Calabi-Yau case, as a special case of our Corollary \ref{cor:main1} above, is totally unknown. Our approach is independent of the structure theorem and it provides a uniform treatment for problems of such types. Indeed, we can prove the following slightly more general result for the compact K{\"a}hler manifold $X$ with a pseudoeffective anticanonical bundle $-K_X$ under a natural assumption of the singularities of $c_1(-K_X)$. 

\begin{theorem}  
\label{thm:main1} 
Let $X$ be a compact K{\"a}hler manifold. Suppose $-K_X$ is pseudo-effective and asymptotically klt. Then the fundamental group $\pi_1(X)$ is almost Abelian. 
\end{theorem}

It is easy to see that case (2) of Corollary \ref{cor:main1} is indeed a special case of Theorem \ref{thm:main1} (see Lemma \ref{1-2}). Here, {\itshape asymptotically klt} classes can be regarded as an analytic counterpart for (or in some sense, a generalization of) effective $\mathbb{Q}$-divisors with log canonical singularities, as originally considered in birational geometry and we postpone its precise definition to Definition \ref{def:asym_klt}.

%Suppose $X$ is projective in Corollary \ref{cor:main1}, then $\pi(X)$ is almost abelian by the same argument in \cite{}. Naturally, we propose the following conjecture in the singular setting. 
%
%\begin{conjecture}\label{conj:main1} Let $X$ be a normal projective variety. Suppose $\Delta$ is an effective $\mathbb{Q}$-divisor such that the pair $(X, \Delta)$ is klt and $-(K_X + \Delta)$ is nef. Then $\pi(X)$ is almost abelian. 
%\end{conjecture}

%In higher dimensions, one can show that Conjecture \ref{conj:main1}  holds if $(X, \Delta)$ is smoothable.  

The third main result of our paper generalizes the Qi Zhang's surjectivity theorem for the Albanese maps of log canonical pairs with nef anti-log canonical divisor to possibly singular complex varieties in Fujiki class $\mathscr{C}$. It serves as one of the main ingredients for the proof of Theorem \ref{thm:main0} and is also of independent interest.

\begin{theorem}
\label{thm:Alb}
Let $X$ be compact normal complex variety in Fujiki class $\mathscr{C}$. Assume that there is an effective $\R$-divisor $\Delta$ on $X$ such that $(X,\Delta)$ is log canonical and $-(K_X+\Delta)$ is nef. Then any smooth model of the Albanese map $\alb_X: X\dashrightarrow\Alb_X$ of $X$ is a fibration, namely, a proper surjective morphism with connected fibres.
\end{theorem}

Here, a compact complex variety is said to be in Fujiki class $\mathscr{C}$ if it is bimeromorphic to a compact K{\"a}hler manifold; or equivalently, by \cite[Th{\'e}or{\`e}me 3]{Var86}, if it is the image of a compact K{\"a}hler manifold. Theorem \ref{thm:Alb} can be regarded as the generalization of the following results:
\begin{itemize} 
\item If $X$ is smooth and $\Delta=0$, then the surjectivity of $\alb_X$ is proved in \cite[Theorem 1.4]{Pau17}.
\item If $X$ is projective (but not necessarily smooth), then Theorem \ref{thm:Alb} is established in \cite{Zha05} (see also \cite{Zha96}). Further results (e.g. flatness and semistability) are proved in \cite[Lemma 4.1]{Wan22} ($X$ is assumed to be Cohen-Macaulay), \cite[Proposition 4.1]{EIM23} ($X$ is assumed to be smooth) and \cite[Theorem 1.3]{BFPT24}.
\item If $(X,\Delta)$ is a klt pair, then Theorem \ref{thm:Alb} is proved in \cite[Corollary 3.13]{MWWZ25}
\item If $(X,\Delta)$ is log smooth, then the surjectivity of $\alb_X$ is shown in \cite[Theorem 1.3]{FHZ25} by establishing a positivity theorem for the variation of fibrewise K{\"a}hler-Einstein metrics. 
\end{itemize}

Let us remark that as long as the surjectivity of $\alb_X$ is proven, one can easily deduce the connectedness of the fibres from \cite[Theorem C]{Wan21}. 

Moreover, by applying the same argument we generalize Chen-Zhang's theorem \cite[Main Theorem]{CZ13} and Deng's theorem \cite[Theorem B]{YDeng21}, addressing a classical question of Demailly-Peternell-Schneider \cite[Problem 4.13.a]{DPS01}, see Theorem \ref{thm:Chen-Zhang-Deng}.

\subsection{Strategy and main idea of the proof}
First, let us briefly explain P{\u a}un's proof of Conjecture \ref{conj:main} (or \ref{conj:main1}) in the case $X$ smooth and $\Delta=0$. It is based on the following 3 main ingredients:
\begin{itemize}
\item[(a)] The Margulis lemma (see \cite[Theorem, p.1250]{Pau97} or Theorem \ref{magu}), which is established by Cheeger-Colding \cite{CC96} and generalized by \cite{KW}.
\item[(b)] The construction of a family of K{\"a}hler metrics $\omega_\epsilon$ in a fixed K{\"a}hler class such that $\ric(\omega_\epsilon)\geq -\epsilon\omega_{\epsilon}$, and this can be done by Aubin-Yau's theorem \cite{Aub78,Yau78}.  
\item[(c)] The surjectivity of the Albanese map, which is established in the projective case by \cite{Zha96} and generalized to K{\"a}hler case by \cite{Pau01,Pau17}. (c) combined with \cite[Corollaire 3.1]{Cam95} implies that $\pi_1(X)$ is almost Abelian if and only if it is virtually nilpotent. 
\end{itemize}
(b) combined with (a) implies that there exists a uniform constant $r$, only depending on the dimension of $X$, such that the image of the fundamental group of a geodesic ball of radius $r$ in $\pi_1(X)$ is virtually nilpotent. And the key observation of \cite{Pau97} is the following
\begin{itemize}
\item[(d)] for an appropriately chosen point $x\in X$ and $\epsilon$ sufficiently small, the inclusion of the geodesic ball of radius $r$ with respect to some $\omega_\epsilon$ induces a surjection of fundamental groups.
\end{itemize}

In this paper we follows the same line as \cite{Pau97} to establish Theorem \ref{thm:main0}, Corollary \ref{cor:main1} \ref{thm:main1} and \ref{thm:main2}. To this end, it suffices to generalize (a)(b)(c)(d) to the case of lc pairs. The generalization of (d) is somehow direct, and the main observation is that singularities do not cause essential trouble to generalize the key lemma \cite[Lemma 1.3]{DPS93}. (c) is proved for the klt paris in \cite[Corollary 3.13]{MWWZ25} based on the recent works of Claudon-H{\"o}ring \cite{CH24} and of Ou \cite{Ou25}, and in this paper we will establish the lc case by a similar argument, see \cite{FHZ25} for the case that $(X,\Supp(\Delta))$ is log smooth. As for (a) and (b), the situation turns out to be quite subtle: a direct generalization seems impossible, at least within the category of (complete) Riemannian manifolds. Indeed, people have already been aware of this kind of difficulty which appear ubiquitously in the study of singular varieties from a differential-geometric viewpoint, not only in the study of fundamental groups, and this is why the notion of RCD spaces are introduced (see \cite{Gig15}), as natural extension (in the metric sense) of complete Riemannnian manifolds with Ricci curvature bounded below. For compact RCD spaces, the Margulis lemma (a) is established by \cite{DSZZ}; and the remaining difficulty is to show that a compact K{\"a}hler space $X$ satisfying the conditions of Conjecture \ref{conj:main} is a K{\"a}hler RCD space. By the work of \cite{EGZ} we can establish (b) for lc pairs in a weak sense, namely, construct a series of K{\"a}hler currents $\omega_\epsilon$ with $\ric(\omega_\epsilon)\geq-\epsilon\omega_\epsilon$ in the sense of current; yet it turns out to be a difficult task to deduce from this that $X$ is a K{\"a}hler RCD space (see \cite[Conjecture 1.1]{FGS25}), and this deduction is only known in the case of Corollary \ref{cor:main1}, Theorem \ref{thm:main1} and Theorem \ref{thm:main2}.

The rest part of the paper is organized as follows: In section \ref{sec:nilpotency}, we prove that $\pi_1(X)$ is almost (or virtually) nilpotent under the assumption of Theorems \ref{thm:main0}, \ref{thm:main1} and Corollary \ref{cor:main1}. In section \ref{sec:abelianity} we prove Theorem \ref{thm:Alb}, which allows us to conclude the proof; based on this argument, we also establish the specialness of the varieties in Theorem \ref{thm:Alb}. 
   %%%%%%%%%%%%%%%%%%%%%%%%%%%%%%%%%%%%%%%%%%%%

%

  %%%%%%%%%%%%%%%%%%%%%%%%%%%%%%%%%%%%%%%%%%%%

  %\section{Margulis Lemma for RCD spaces }
 \bigskip

\noindent\textbf{Acknowledgment:} We would like to thank Professor Henri Guenancia, Mihai P{\u a}un, Chenyang Xu for very helpful conversations. We are grateful to Professor Jiayin Pan for bringing \cite{DSZZ} to our attention and also grateful to Professor Fr{\'e}d{\'e}ric Campana for pointing to us the relation of our main theorem to his Abelianity conjecture, and to Professor Masataka Iwai for reminding us about the the generic nefness theorem for the tangent sheaves in our case and showing us the deduction of the specialness from it.

%Xin Fu is  supported by National Key R\&D Program of China 2024YFA1014800 and NSFC No. 12401073. JW is partially supported by NSFC (Grants No.12301060 and 12288201) and National Key R\&D Program of China (Grant No. 2021YFA1003100).

%%%%%%%%%%%%%%%%%%%%%%%%%%%%%%%%%%%%%%%%%%%%
 \section{Proof of virtual nilpotency}
 \label{sec:nilpotency}
  In this section, we proceed to prove Theorem \ref{thm:main1}, Theorem \ref{thm:main0} and Corollary \ref{cor:main1}. In particular, we prove that the fundamental group $\pi_1(X)$ under the assumption of these theorems are virtually nilpotent and leave the proof of almost abelian property to next section. As mentioned in the introduction, in each case, the idea is the same: we  are  able to construct a sequence of Ricci almost non negative RCD spaces which are homeomorphic to $X$, then me way apply the RCD version of Margulis lemma \ref{DSZZ} to prove the fundamental group is virtual nilpotent.
 \subsection{Asymptotically klt class} 
 In this subsection, we recall the following definition of asymptotically klt classes on a compact K\"ahler manifold and give some remarks about this class. Intuitively, these classes can be regarded as an analytic counterpart for generalized log canonical pairs (see \cite[Definition 4.1]{BZ16}).
\begin{definition}
\label{def:asym_klt}
Let $X$ be a compact K{\"a}hler manifold and let $\alpha\in H^{1,1}(X,\R)$ be a pseudoeffective class (i.e. containing a semipositive closed current). Choose a K{\"a}hler form $\omega$ on $X$, and consider the set $\alpha[-\epsilon\omega]$ of all the almost positive closed $(1,1)$-currents $T\in\alpha$ satisfying $T\ge -\epsilon\omega$ and let $T_{\min,\epsilon}\in\alpha[-\epsilon\omega]$ be one with minimal singularities $($see e.g., \cite[\S 2.8]{Bou04}$)$. By writing $T_{\min,\epsilon}=\theta_\epsilon+dd^c\varphi_\epsilon$ with $\theta_\epsilon\in\alpha$ being a smooth $(1,1)$-form and $\varphi_\epsilon$ being a quasi-psh function, we define the $\epsilon$-asymptotic multiplier ideal of $\alpha$ (with respect to $\omega$) to be
\[
\mathscr{J}_{\epsilon,\omega}(\alpha):=\mathscr{J}(\varphi_\epsilon).
\]
The definition does not depend on the choice of $T_{\min,\epsilon}$ since any two currents with minimal singularities have equivalent singularities and thus the same multiplier ideal sheaves. We simply write $\mathscr{J}(\alpha)$ for $\mathscr{J}_{0,\omega}(\alpha)$ as it apparently does not depend on $\omega$.

We say that the class $\alpha$ is asymptotically klt if 
\[
\mathscr{J}_{\epsilon,\omega}(\alpha)=\mathscr{O}_X 
\]
for every $\epsilon>0$. Clearly this definition does not depend on the choice of the K{\"a}hler form $\omega$. Similarly, we say that an $\R$-line bundle $L$ is asymptotically klt if its Chern class $c_1(L)$ is so. 
\end{definition}
\begin{remark}
\label{rmk:asym_klt}
$($a$)$ We can reformulate the definition of asymptotically klt classes as follows. For any K{\"a}hler class $\omega$, we have $\mathscr{J}_{\epsilon_1,\omega}(\alpha)\supseteq\mathscr{J}_{\epsilon_2,\omega}(\alpha)$ for $\epsilon_1\ge\epsilon_2$, thus we can define the {\itshape asymptotic multiplier ideal} of $\alpha$ to be 
\[
\mathscr{J}_+(\alpha):=\bigcap_{\epsilon>0}\mathscr{J}_{\epsilon,\omega}(\alpha),
\]
which is independent of the K{\"a}hler form $\omega$ chosen. Clearly we have $\mathscr{J}(\alpha)\subseteq\mathscr{J}_+(\alpha)$, and the inclusion is in general not an equality $($just like log canonical singularities$)$. And it is easy to see that $\alpha$ is asymptotically klt if and only if $\mathscr{J}_+(\alpha)=\mathscr{O}_X$. 

$($b$)$ By definition, nef classes are asymptotically klt. Moreover, being asymptotically klt is a regularity condition posed for pseudoeffective classes that can be regarded as a singular version of the nefness. Indeed, we can prove that the non-nef locus (see \cite[Defintion 3.3]{Bou04}) of a pseudoeffective class $\alpha\in H^{1,1}(X,\R)$ satisfies
\begin{equation}
\label{eq:non-nef-locus}
\NNef(\alpha)=\bigcup_{m\in\Z_{>0}}V(\mathscr{J}_+(m\alpha))_{\red},
\end{equation}
here, for any ideal sheaf $\mathscr{I}\subseteq\mathscr{O}_X$, $V(\mathscr{I})$ denotes the complex subspace defined by $\mathscr{I}$. For a general account of non-nef loci, see Appendix \ref{app:nonnef}.

$($c$)$ Let $L$ be a pseudoeffective $\mathbb{R}$-line bundle on $X$, and fix a K{\"a}her form $\omega$ on $X$. Then $L$ is asymptotically klt if and only if, for every $\epsilon>0$, $L$ admits a possibly singular Hermitian metric $h_\epsilon$ such that $\sqrt{-1}\Theta_{h_\epsilon}(L)\ge -\epsilon\omega$ and $\mathscr{J}(h_\epsilon)=\mathscr{O}_X$.
\end{remark}

The following lemma show that asymptotic klt classes can be regarded as the analytic counterpart of generalized log canonical pairs in the bimeromorphic sense. Indeed, given a generalized log canonical pair $(X,B+\mathbb{M})$, we can find a smooth model $Y\to X$ such that the condition $B_Y^{\ge0}+M_Y$ satisfies the condition of the lemma. 
\begin{lemma} \label{1-2}
Let $X$ be a compact K{\"a}hler manifold and $\Delta$ an effective $\R$-divisor such that the pair $(X, \Delta)$ is log canonical (lc). Then $\eta+\{\Delta\}$ is pseudoeffective and asymptotically klt for any nef class $\eta\in H^{1,1}(X,\R)$. 
\end{lemma}
\begin{proof}
Let $\theta$ be a K{\"a}hler form on $X$, then for any $\epsilon>0$ the class $\eta+\epsilon\cdot\{\theta\}$ is K{\"a}hler thus so is $\eta+\epsilon\cdot\{\theta\}+\delta(\epsilon)\cdot\{\Delta\}$ for $\delta(\epsilon)>0$ sufficiently small since the K{\"a}hler cone is open. 
\end{proof}

\subsection{K{\"a}hler RCD space with almost nonnegative Ricci curvature }
 
 In this subsection, in each set up of Theorem \ref{thm:main1}, Corollary \ref{cor:main1}, we construct a sequence of Ricci almost non negative RCD spaces which are homeomorphic to $X$.

\subsubsection{Asymptotic klt case:}
Firstly, we treat the case of Theorem \ref{thm:main1}. More precisely, we prove the following:

\begin{theorem}\label{zz} Let $(X, \theta)$ be an $n$-dimensional K{\"a}hler manifold equipped with a smooth K{\"a}hler form $\theta$. Suppose that $-K_X$ is pseudoeffective and asymptotically klt.
Then for any $\epsilon>0$, there exists a K{\"a}hler current $\omega_\epsilon \in [\theta]$ satisfying the following. 
\begin{enumerate}

\item $\omega_\epsilon$ has bounded local potential and is smooth on $X\setminus Z_\epsilon$ for some analytic subvariety $Z_\epsilon$ of $X$.

\smallskip

\item Let $(X_\epsilon, d_\epsilon, \mu_\epsilon)$ be the metric completion of $(X\setminus Z_\epsilon, \omega_\epsilon, \omega_\epsilon^n)$. Then $(X_\epsilon, d_\epsilon, \mu_\epsilon)$ is a compact $\RCD(-\epsilon, 2n)$-space homeomorphic to $X$.

\end{enumerate}

\end{theorem}

Now we proceed to the proof. Without loss of generality, we can assume that $\theta + \ric(\theta)>0$  after a fixed rescaling $\theta$.  For each $\epsilon>0$, by our assumption there exists $f_\epsilon\in \PSH(X, \ric(\theta)+ \epsilon \theta)$ such that $\mathscr{J}(f_\epsilon)=\mathscr{O}_X$. Then by openness \cite{Ber13,GZ15} we have the following lemma. 

\begin{lemma} \label{lem41}For any $\epsilon>0$, there exist $p>1$ and $f_\epsilon \in \PSH(X, \ric(\theta)+ \epsilon \theta)$ such that 
$$e^{-f_\epsilon}\in L^p(X). $$  

\end{lemma}
\begin{proof} By our assumption, $-K_X$ is pseudoeffective and asymptotically klt, thus there is a $(\ric(\theta)+\epsilon\theta)$-psh function $f_\epsilon$ such that $e^{-f_\epsilon}$ is $L^1$-integrable. By Demailly-Koll{\'a}r's openness conjecture solved by \cite{Ber13,GZ15}, $e^{-f_\epsilon}$ is $L^p$ integrable for some $p>1$. The lemma is thus proved.
  \end{proof}
\begin{lemma} \label{lem42} For any $\epsilon>0$, there exists $F_\epsilon \in \PSH(X, \ric(\theta)+ \epsilon \theta)$ satisfying the following 

\begin{enumerate}

\item $F_\epsilon\in C^\infty(X\setminus \cS)$ for some analytic subvariety $\cS$ of $X$ and $F_\epsilon$ has analytic singularities.

\smallskip

\item $e^{-F_\epsilon}\in L^p(X)$ for some $p>1$.

\end{enumerate}

\end{lemma}

\begin{proof} We fix $f_{\epsilon/2}$ as chosen in Lemma \ref{lem41}. By Demailly's regularization \cite[Theorem (13.12), p.134]{Dem10}, there exist $f_{k, \epsilon}\in \PSH(X, \ric(\theta)+ 2\epsilon \theta)$ such that 

\begin{enumerate}

\item $f_{k, \epsilon/2}$ is smooth on $X\setminus Z_k$ of an analytic set $Z_k$ with 
$$Z_0\subset Z_1\subset ... \subset Z_k \subset ...\subset X. $$

\item $f_{k, \epsilon/2}$ has analytic singularities and $f_{k, \epsilon/2}$ decreasingly converge to $f_{\epsilon/2}$.  

\end{enumerate}
The lemma is then proved by choose $F_\epsilon = f_{k, \epsilon/2}$ for sufficiently large $k>0$. 
\end{proof}

We consider the following complex Monge-Amp{\`e}re equation 
\begin{equation}\label{eqn1}
(\theta + \ddbar \varphi_\epsilon)^n = e^{\epsilon\varphi_\epsilon - F_\epsilon }\theta^n. 
\end{equation}
Its induced curvature equation for $\omega_\epsilon = \theta + \ddbar \varphi_\epsilon$ is given by
\begin{equation}\label{ricepsilon}
\ric(\omega_\epsilon) = - \epsilon \omega_\epsilon + (\ric(\theta) + \epsilon \theta + \ddbar F_\epsilon) \geq - \epsilon \omega_\epsilon 
\end{equation}
in the sense of currents.

Ko\l{}odziej's theorem \cite{Kol} combined with Lemma \ref{lem42} immediately gives the following $L^\infty$-estimate. 

\begin{lemma} Fix $\epsilon\in (0, 1)$ and let $Z_\epsilon$ be the analytic subvariety of $X$ associated to the singularities of $F_\epsilon$, then there exists a unique $\varphi_\epsilon \in \PSH(X, \theta) \cap L^\infty(X) \cap C^\infty(X^\circ)$ solving \eqref{eqn1}, where $X^\circ:=X\setminus Z_\epsilon$. 
\end{lemma}

 %We have the following lemma by standard regularization for $F_\epsilon$ along with the $L^\infty$-estimate and the Schwarz lemma.

%\begin{lemma} There exists a subvariety $Z_\epsilon$ of $X$ such that 
%
%
%$$\varphi_\epsilon \in C^\infty(X\setminus Z_\epsilon). $$
%

%
%\end{lemma}

We define 
$$(\hat X_\epsilon, d_\epsilon, \mu_\epsilon) = \overline{(X\setminus Z_\epsilon, \omega_\epsilon, \omega_\epsilon^n)}$$ as the metric completion of $X\setminus Z_\epsilon$.

If we let 
$$\psi_\epsilon  = \varphi_\epsilon + F_\epsilon \in PSH(X, \ric(\theta) + (1+\epsilon)\theta),$$ 
then equation (\ref{eqn1}) is equivalent to 
\begin{equation}\label{eqn2}
(\theta+ \ddbar \varphi_\epsilon)^n = e^{(1+\epsilon) \varphi_\epsilon - \psi_\epsilon}\theta^n. 
\end{equation}

\begin{lemma}\cite[Lemma 2.4]{FGS25} For each $\epsilon\in (0,1)$, there exist $\psi_{\epsilon, j} \in \PSH(X, \ric(\theta) +(1+\epsilon) \theta)\cap C^\infty(X)$ such that $\psi_{\epsilon, j}$ converges decreasingly to $\psi_\epsilon$ as $j\rightarrow\infty$. Furthermore, we can assume that $\psi_{\epsilon, j}$ converge smoothly to $\psi_\epsilon$ away from $Z_\epsilon$.
%\textcolor{red}{can we?the following argument works}

\end{lemma}

\begin{proof}Recall that, we have assuemd that $\ric(\theta) + \theta$ is K\"ahler. By \cite{BK07},  there exists $\psi_i \in \PSH(X, \ric(\theta) + (1+\epsilon)\theta) \cap C^\infty(X)$ such that $\psi_i$ converge to $\psi_\epsilon$ decreasingly. Since $\psi_\epsilon$ is smooth on $X\setminus Z_\epsilon$, $\psi_i$ converges to $\psi_\epsilon$ in $L^\infty(\cK)$ for any $\cK\subset\subset X\setminus Z_\epsilon$. 

We would like to modify $\psi_i$ so that it converges smoothly on $\cK$.  Note that $\varphi_\epsilon$ is bounded and $F_\epsilon(x)$ goes to $-\infty$ as $x\to Z_\epsilon,$ so $\psi_\epsilon\in\PSH(X, \ric(\theta) + (1+\epsilon)\theta)$ and goes to $-\infty$ when $x\rightarrow Z_\epsilon$. Let $\tilde \psi_{i, \tau, \delta} = \mathcal{M}_{\tau} (\psi_i, \psi_\epsilon + \delta )$ be the regularized maximum of $\psi_i$ and $\psi_\epsilon+\delta $  for $\delta, \tau>0$ (c.f. \cite[Lemma I.5.18, p.43]{agbook}). By definition, we have $\tilde\psi_{i, \tau, \delta}\in \PSH(X, \ric(\theta) + (1+\epsilon)\theta)$ with $\tilde\psi_{i, \tau, \delta}\geq \psi_i\geq \psi$ . Furthermore, $\tilde\psi_{i, \tau, \delta} \in C^\infty(X)$ since $\tilde \psi_{i, \tau, \delta} = \mathcal M_\tau(\psi_i)$ near $X\setminus Z_\epsilon$ for sufficiently small $\tau>0$. 

For any $K\subset\subset X\setminus Z_\epsilon$,  $\tilde \psi_{i, \tau, \delta}= \mathcal M_\tau(\psi_\epsilon  + \delta)$ on $K$ for sufficiently large $i>0$ and sufficiently small $\delta\gg\tau$, since $\psi_i$ converges to $\psi$ uniformly in $L^\infty(K)$.  At the same time, $\tilde \psi_{i, \tau, \delta} = \mathcal M_\tau(\psi_i)$ near $Z_\epsilon$ as $\psi_\epsilon+\delta $ tends to $-\infty$ along $X\setminus Z_\epsilon$. By choosing suitable $\tau_i, \delta_i \rightarrow 0$, $\tilde \psi_{i, \tau_i, \delta_i}$ converges to $\psi$ smoothly on any fixed compact subset of $X\setminus Z_\epsilon$. This proves the lemma.
\end{proof}
%\textcolor{red}{I think the argument is OK. Yet, there is one point that I am bit confused: since $X$ is smooth, the convolution is always smooth, it seems that there is no need to add the term $\delta^2F_\epsilon$ to control the behvaiour near $Z_\epsilon$. }

We can consider the following complex Monge-Amp{\`e}re equation
\begin{equation}\label{eqn3}
(\theta + \ddbar \varphi_{\epsilon, j})^n = e^{(1+\epsilon)\varphi_{\epsilon, j} - \psi_{\epsilon, j} }\theta^n.
\end{equation}
The corresponding curvature equation for $\omega_{\epsilon,j} = \theta + \ddbar \varphi_{\epsilon, j}$ is given by
\begin{equation}\label{ric}\ric(\omega_{\epsilon,j}) = - (1+\epsilon)  \omega_{\epsilon,j} + (\ric(\theta) + (1+\epsilon) \theta + \ddbar \psi_{\epsilon, j}) \geq -(1+\epsilon) \omega_{\epsilon, j}. \end{equation}
\begin{lemma} For any $\epsilon\in (0,1)$, there exist $p>1$ and $C>0$ independent of $j$ such that for all $j>0$ we have
$$\|e^{-\psi_{\epsilon, j}}\|_{L^p(X)} \leq C, $$
$$\|\varphi_{\epsilon, j} \|_{L^\infty(X)} \leq C $$
and
$$\diam(X, \omega_{\epsilon, j}) \leq C. $$
In addition, the Ricci curvature $\omega_{\epsilon, j}$ is uniformly bounded below by $-2$. 

\end{lemma}
\begin{proof} Note that $\psi_{\epsilon,j}\geq \psi_\epsilon$ and the $L^p$ bound of $e^{-\psi_\epsilon}$ follows from the boundedness of $F_\epsilon$ and $\varphi_\epsilon$ \cite{Kol}. So $e^{-\psi_{\epsilon,j}}$ is uniformly $L^p$ bounded. Then the uniform $L^\infty$ bound of $\varphi_{\epsilon,j}$ is standard by \cite{Kol}. The diameter bound of $\omega_{\epsilon,j}$ follows from \cite{GPSS2}.
\end{proof}
\begin{lemma}\label{Sch}
There is a constant $C>0$ independent of $j$ such that
$$\omega_{\epsilon,j}\geq C\theta.$$ Moreover, for any compact set $K\subset\subset X\setminus Z_\epsilon$, $\varphi_{\epsilon, j}$ converge smoothly to $\varphi_\epsilon $ on $X\setminus Z_\epsilon$ after passing to a subsequence. 
\end{lemma}

\begin{proof} By the Ricci lower bound \eqref{ric}, $\omega_{\epsilon,j}\geq C\theta$ follows from Yau's $C^2$ estimate. Note $\psi_{\epsilon,j}$ converge smoothly to $\psi_\epsilon$ on $X\setminus Z_\epsilon$. This implies that $\omega_{\epsilon,j}\leq C_K\theta$ for some positive constant $C_K$. Then higher order estimates will follow from Evans-Krylov theory. 
\end{proof}

\begin{lemma}[{\bfseries =Proof of Theorem \ref{zz}}]
\label{A} 
$(X, \omega_{\epsilon, j}, \omega_{\epsilon,j}^n)$ converges to $( \hat X_\epsilon, d_\epsilon, \mu_\epsilon)$ in Gromov-Hausdorff topology as $j\rightarrow \infty$ with  $\hat X_\epsilon$ homeomorphic to $X.$
Moreover,  $(\hat X_\epsilon, d_\epsilon, \mu_\epsilon)$ is an $\RCD(-\epsilon, 2n)$-space.

\end{lemma}
\begin{proof} By the Ricci lower bound \eqref{ric}, and uniform diameter upper bound \cite{GPSS2} (independent of $j$), standard Cheeger-Colding theory implies that up to a subsequence $(X, \omega_{\epsilon, j}, \omega_{\epsilon,j}^n)$ converges to a compact metric measure space in Gromov-Hausdorff topology  as $j\rightarrow \infty$. Since $Z_\epsilon$ is an analytic subvariety of $X$ of Hausdorff codimension no less than $2$, $X\setminus Z_\epsilon$ is uniformly almost convex in $(X, \omega_{\epsilon, j})$ (\cite{Gromov}). Therefore the limiting space  must coincide with $(\hat X_\epsilon, d_\epsilon, \mu_\epsilon)$. It is clear that $\hat X_\epsilon\in \RCD(-2,2n)$ and moreover it is an almost smooth metric measure space in the sense of Honda \cite{Ho}. Then by the Ricci lower bound on   $X \setminus Z_\epsilon$ \eqref{ricepsilon},  we can improve the global Ricci lower bound to $-\epsilon$ and so $ \hat X_\epsilon \in \RCD(-\epsilon, 2n)$ (c.f. \cite[Lemma 2.3]{FGS25} for more details).  

By Lemma \ref{Sch}, the identity map $\iota_j: (X,\omega_{\epsilon,j})\rightarrow (X,\theta)$ is a sequence of uniform Lipschitz map. Furthermore, $\iota$ will converge to a Lipschitz surjective map $\iota_\infty: \hat X_\epsilon \to X$. 

Next we will show that $\iota_\infty$ is injective by a local version of the fundamental partial $C^0$-estimate of Tian \cite{Tia90} (see also the fundamental work of Donaldson-Sun \cite{DS14}). One cannot directly apply the global partial $C^0$-estimate since $[\theta]$ is not necessarily a rational class. However, due to the Schwarz lemma \cite[Theorem 1.2]{CCHSTT25},  if we fix an open Euclidean ball $B \subset X$ (in particular, it is strongly pseudoconvex), $B$ contains a geodesic ball of $\omega_{\epsilon, j}$ for $j$ sufficiently large. So each $\omega_{\epsilon, j}$ is the curvature of a trivial line bundle $L_j$ on $B$ and $(B, \omega_{\epsilon,j})$ converge locally to $(B_\epsilon, d_\epsilon) \subset \hat X_\epsilon$.  By \cite{LS1}, each tangent cone of $(B_\epsilon, d_\epsilon)$ is good in the sense of \cite{CDSII}. We can immediately apply the partial $C^0$-estimates on $B$ in \cite{LS1} as a quantitative Kodaira embedding theorem. In particular, for any two different points $p,q\in B_\epsilon$, there is a holomorphic function $\sigma$ on $B$ such that $\iota_\infty^*\sigma(p)\neq \iota_\infty^*\sigma(q)$. This implies that $\iota_\infty(p)\neq \iota_\infty(q)$. Hence $\iota_\infty$ is injective on $\iota_\infty^{-1}(B)$, hence injective on $X_\epsilon$.

Finally, we remark that $X\setminus Z_\epsilon$ lies in the regular part of $\hat X_\epsilon$ as $\omega_{\epsilon,j}$ converges smoothly to $\omega_\epsilon$ away from $Z_\epsilon$. 
\end{proof}

%%%%%%%%%%%%%%%%%%%%%%%%%%%%%%%%%%%%%%%%%%%%%

%\textcolor{red}{XIN: REWRITE A LITTLE BIT
%On the other hand, to further exhibit our theme in this paper: we can derive the nilpotency property of the fundamental groups of varieties under consideration in this note by constructing a sequence of Ricci almost nonnegative RCD space, we treat the klt pair case of Theorem \ref{cor:main1} by constructing conic type metrics. We fix the set-up as follows:} 

%Suppose $X$ is an $n$-dimensional K{\"a}hler manifold and 

%Let $\omega_0 \in \alpha$ be a fixed K{\"a}hler metric on $X$. Since $-(K_X + \Delta)$ is nef, there exists $f\in PSH(X, \ric(\omega_0)- \ric(h))$ such that $f$ has minimal singularities. In particular, the Lelong number of $f$ vanishes everywhere.

%Since $-(K_X+\Delta)$ is klt, there exists $p>1$ such that $  |\sigma|_h^{-2} \in L^p(X, (\omega_0)^n). $ Hence

%\begin{lemma} For all $\epsilon\in (0,1)$, there exists a unique solution $\varphi_\epsilon \in \PSH(X, \omega_0)\cap L^\infty(X)\cap C^\infty(X\setminus \Delta)$ to equation (\ref{cma1}).  
%\end{lemma}
%Similarly, we define 
%
%$$(\hat X_\epsilon, d_\epsilon, \mu_\epsilon) = \overline{(X\setminus \Delta ,\omega_\epsilon, \omega_\epsilon^n)}$$ as the metric completion of $X\setminus \Delta$. 

%\begin{proposition}\label{B}   $\hat X_\epsilon$ homeomorphic to $X.$
%Moreover,  $(\hat X_\epsilon, d_\epsilon, \mu_\epsilon)$ is an $\RCD(-\epsilon, 2n)$-space. 
%\end{proposition}
%\begin{proof} The proof is the same as the proof of Lemma \ref{A}.
%\end{proof}
\subsubsection{complex dimension three case:}
We proceed to construct a sequence of RCD spaces with almost nonnegative Ricci curvature under the assumption of  Corollary \ref{cor:main1}.  By Lemma \ref{1-2}, item $(2)$ of Corollary \ref{cor:main1} is a special case of Theorem \ref{thm:main1}, so we will focus on item $(1)$ of Corollary \ref{cor:main1} from now on.  We will first assume $\Delta=0$. Let $X$ be a $3$-dimensional normal K\"ahler variety with log terminal singularities. We let $X^\circ$ be the smooth part of $X$. Suppose $-K_X$ is nef.  Let $\theta$ be a smooth K{\"a}hler form and $\Omega$ be a smooth adapted volume measure on $X$. 
%We may also assume that the class of $\theta- K_X$ is K\"ahler and 
 %
 %$$\ric(\Omega) + \theta$$
 %
 %is a smooth K{\"a}hler metric on $X$. 
 %
 Since $-K_X$ is nef, for any $\epsilon\in (0,1)\cap \mathbb{Q}$, $c_1(X) +\epsilon[\theta]$ is K\"ahler class and so there exists $f_\epsilon \in \PSH(X, \ric(\Omega) + \epsilon \theta)\cap C^\infty(X)$.

 We will consider the following complex Monge-Amp{\`e}re equation
\begin{equation}\label{cma3d}
(\theta + \ddbar \varphi_\epsilon)^n = e^{\epsilon \varphi_\epsilon+ f_\epsilon} \Omega. 
\end{equation}

 We let $\omega_\epsilon = \theta + \ddbar \varphi_\epsilon$.  Then the curvature equation induced by equation (\ref{cma3d}) is given by 
 \begin{equation}\label{curveq3d}
 \ric(\omega_\epsilon) = -\epsilon \omega_\epsilon + \left( \ric(\Omega) + \epsilon \theta + \ddbar f_\epsilon \right) \geq -\epsilon \omega_\epsilon.
 \end{equation}

 Then the following regularity result for solution to \eqref{curveq3d} is standard (c.f. \cite{EGZ}). 
 \begin{lemma} For any $\epsilon \in (0,1) \cap \mathbb{Q}$, there exists a unique $\varphi_\epsilon \in \PSH(X, \theta) \cap L^\infty(X) \cap C^\infty(X^\circ)$. 
 
 \end{lemma}
Now $\omega_\epsilon$ is a smooth K{\"a}hler metric on $X^\circ$ and we define
 $$(\hat X_\epsilon, d_\epsilon, \mu_\epsilon) = \overline{(X^\circ, \omega_\epsilon, \omega_\epsilon^n)}.$$
 as the metric completion of the metric measure space $(X^\circ, \omega_\epsilon, \omega_\epsilon^n)$. The following regularity result of $(\hat X_\epsilon, d_\epsilon, \mu_\epsilon)$  is established in \cite[Theorem 1.1]{FGS25}. 

\begin{proposition}[\cite{FGS25}] 
\label{thm:FGS}
For any $\epsilon \in (0,1) \cap \mathbb{Q}$,   $(\hat X_\epsilon, d_\epsilon, \mu_\epsilon)$ is a compact $\RCD(-\epsilon, 2n)$-space homeomorphic $X$. In particular,
\begin{enumerate}

\item  the regular part of $\hat X_\epsilon$ coincides with $X^\circ$,  
 
 \smallskip
 
\item the identity map $\iota: (X^\circ, \omega_\epsilon) \rightarrow (X^\circ, \theta)$ extends to a surjective Lipschitz map from $ (\hat X_\epsilon, d_\epsilon)$ to  $(X, \theta)$.

\end{enumerate}

 \end{proposition}

 The following remark is very important.

\begin{remark}\label{rmk}In \cite[Theorem 1.1]{FGS25}, we need to assume that $X$ is projective and $\theta$ is a polarized K\"ahler metric. We point out the places where the projectivity of $X$ is used. 

\begin{enumerate}
\item In \cite[Lemma 2.4]{FGS25}, we need to regularize the possibly singular positive current $\ric(\Omega) + \epsilon \theta + \ddbar f_\epsilon$ in \eqref{curveq3d}. 
\item We have used MMP in complex dimension three to show that terminization for a klt variety exists. 

\item We use three dimensional terminal singularity is locally $\mathbb Q$ smoothable. 
\item In \cite[Lemma 5.15]{FGS25}, we use sections of ample line bundle $L$ on $X$ to separate points and further show that $\iota$ is homeomorphic.
\end{enumerate}
Below, we deal with the above four items one by one.
\begin{enumerate}
\item  For the first point, $\ric(\Omega) + \epsilon \theta + \ddbar f_\epsilon$ is smooth in our context, so no regularization is needed. 
\item For the second point, in the K\"ahler case, relative MMP holds in complex dimension three. \item For the third point, an analytic termimal singularity in complex dimension three is locally  $\mathbb Q$-smoothable (c.f. \cite[Page 162, Notation 5.29]{KM98}). 
\item To show that $\iota$ is homeomorphic. Partial $C^0$ type estimate in \cite[Lemma 5.15]{FGS25} can be done locally on a strongly pseudoconvex domain. We refer the readers to Appendix \ref{appI} for details.
\end{enumerate}
In sum, Proposition \ref{thm:FGS} holds without $X$ being projective.

\end{remark}

Now we treat the case of a klt K\"ahler pair $(X,\Delta)$ with $-(K_X+\Delta)$ being nef. Let $\Delta = \sum_{i=1}^I a_i D_i$, where $D_i$ is a prime divisor and $a_i\in (0,1)$, $\sigma_i$ be the defining section of $D_i$ and $h_i$ be a smooth hermitian metric of the line bundle associated to $D_i$. We let $\sigma$ be the multi-valued section satisfying
\begin{equation}\label{AAA}|\sigma|= \prod_{i=1}^I  |\sigma_i|^{a_i}  \quad \textnormal{and} \quad h = \prod_{i=1}^I (h_i)^{a_i}. 
\end{equation}
If $-(K_X+\Delta)$ is nef, then $-(K_X +\Delta) + \epsilon \theta$ is K{\"a}hler for any $\epsilon>0$. Therefore for any $\epsilon>0$, there exists $f_\epsilon\in C^\infty(X)$  such that
\begin{equation}\label{eqn:A}\ric(\Omega)+\ddbar f_\epsilon-\ric(h)\geq -\frac{\epsilon}{2}\theta, ~\sup_X f_\epsilon =0.
\end{equation}
%
%or equivalently, 
%
%$$f_\epsilon \in \PSH(X, \ric(\omega_0) -\ric(h) + \epsilon \omega_0). $$
%
Consider the complex Monge-Amp{\`e}re equation
\begin{equation}\label{eqn:pair}
(\theta + \ddbar \varphi_\epsilon)^n =   |\sigma|_h^{-2} e^{\epsilon \varphi_\epsilon - f_\epsilon}\Omega .
\end{equation}
By the klt assumption, for each fixed $\epsilon>0$,  \eqref{eqn:pair} admits a solution  $\varphi_\epsilon\in C^\infty(X^\circ\setminus \Delta)\cap L^\infty(X)$ (c.f.\cite{EGZ}), where $X^\circ$ is the regular part of $X$. 
Set $\omega_\epsilon = \omega_0 + \ddbar \varphi_\epsilon$. Direct calculation implies that
\begin{equation}\label{eqn:bbb}\ric(\omega_\epsilon) = - \epsilon \omega_\epsilon + \left( \ric(\Omega)-\ric(h) + \ddbar f_\epsilon \right) + [\Delta] \geq - \epsilon \omega_\epsilon \end{equation}
globally in the sense of currents.

 %Now $\omega_\epsilon:=\theta+\ddbar\varphi_\epsilon$ is a smooth K{\"a}hler metric on $X^\circ\setminus\Delta$ satisfying  $\ric(\omega_\epsilon)\geq -\epsilon\omega_\epsilon$. 
 Similarly, we define
 \begin{equation}\label{eqn:rcd}(\hat X_\epsilon, d_\epsilon, \mu_\epsilon) = \overline{(X^\circ\setminus\Delta, \omega_\epsilon, \omega_\epsilon^n)},\end{equation}
 as the metric completion of the metric measure space $(X^\circ\setminus\Delta, \omega_\epsilon, \omega_\epsilon^n)$.

 We will establish the regularity  of $(\hat X_\epsilon, d_\epsilon, \mu_\epsilon)$  in Proposition \ref{Prop:antinefpair}, which extends the result of \cite{FGS25} to pairs. Before that, we remark that if $X$ is smooth, it is easy to establish the RCD property of $(\hat X_\epsilon, d_\epsilon, \mu_\epsilon)$.  

\begin{proposition}\label{B}\cite[Lemma 5.7]{FGS25} Suppose $X$ is smooth, then  $(\hat X_\epsilon, d_\epsilon, \mu_\epsilon)$ is a $\RCD(-\epsilon, 2n)$-space. Moreover,  $\hat X_\epsilon$ homeomorphic to $X.$

\end{proposition}
\begin{proof} 
In \cite[Lemma 5.7]{FGS25}, it is proved that  $(\hat X_\epsilon, d_\epsilon, \mu_\epsilon)$ is a $\RCD(-\epsilon, 2n)$-space.
%The argument is the same as \ref{A}. The point is that in this case, there exists $p>1$ such that$  |\sigma|_h^{-2} \in L^p(X, (\omega_0)^n). $ Hence for all $\epsilon\in (0,1)$, there exists a unique solution $\varphi_\epsilon \in \PSH(X, \omega_0)\cap L^\infty(X)\cap C^\infty(X\setminus \Delta)$ to equation (\ref{eqn:pair}). And one may show that $\hat X_\epsilon$ can be approximated by a sequence of smooth K\"ahler metric with Ricci curvature uniformly bounded below via the trick \eqref{eqn2} and \eqref{eqn3}.  
The rest of the proof is the same as the proof of Lemma \ref{A}.
\end{proof}

\begin{proposition}\label{Prop:antinefpair}
For any $\epsilon \in (0,1) \cap \mathbb{Q}$,   $(\hat X_\epsilon, d_\epsilon, \mu_\epsilon)$ is a compact $\RCD(-\epsilon, 2n)$-space homeomorphic $X$. \end{proposition}
\begin{proof}Now the additional difficulty to deal with equation \eqref{eqn:pair} is that we have one more degenerate conic term $|\sigma|_h^{-2}$, where $|\sigma|_h^{-2}$ is defined as in \eqref{AAA}. 
To apply the result of \cite{FGS25}, the key observation is that there is a natural regularization of the quasi PSH function $|\sigma|_h^{-2}$. 
More precisely, we regularize \eqref{eqn:pair} and consider
\begin{equation}
(\theta + \ddbar \varphi_\epsilon)^n =   (|\sigma|^2_h+\delta)^{-1} e^{\epsilon \varphi_\epsilon - f_\epsilon}\Omega .
\end{equation}
The following computation  is well known
\begin{equation}\label{eqn:reg}
\ddbar\log(|\sigma|^2_h+\delta)=-\ric(h)+\frac{\delta}{|\sigma|^2_h+\delta}\ric(h)+\frac{\delta}{(|\sigma|^2_h+\delta)^2}|\partial \sigma\wedge\overline{\partial \sigma}|^2_h.\end{equation}
Direct calculation shows that 

\[\ric(\omega_{\epsilon,\delta})=-\epsilon\ddbar\varphi_\epsilon+\ric(\Omega)+\ddbar f_\epsilon+\ddbar\log(|\sigma|^2_h+\delta).\]
In particular, if $\delta<<\epsilon$, then by \eqref{eqn:A},\eqref{eqn:reg} and \cite[Theorem 1.2]{CCHSTT25}, we have 
%$$\ric(\omega_{\epsilon,\delta})\geq -\epsilon\omega_{\epsilon,\delta}-\epsilon\theta.$$
%By the Schwarz lemma \cite{CCHSTT25}, 
$$\ric(\omega_{\epsilon,\delta})\geq -C\omega_{\epsilon,\delta},$$
for some $C$ independent of $\delta$. Without lose of generality, we assume that $C=1$.
By Remark \ref{rmk} and \eqref{eqn:reg}, we may apply \cite{FGS25} to conclude that for each fixed $1>>\delta>0$, the metric completion \[(\hat X_{\epsilon,\delta}, d_{\epsilon,\delta}, \mu_{\epsilon,\delta}) = \overline{(X^\circ, \omega_{\epsilon,\delta}, \omega^n_{\epsilon,\delta})}\]
is $\RCD(-1, 2n)$ space and moreover, $\hat X_{\epsilon,\delta}$ is homeomorphic to $X$.

Now by the compactness of $RCD(-1,2n)$, when $\delta\rightarrow 0$, $(\hat X_{\epsilon,\delta}, d_{\epsilon,\delta}, \mu_{\epsilon,\delta})\rightarrow (Y, d_Y,\mu_Y)$ in the Gromov-Hausdorff sense. Also it is clear that the convergence is $C^\infty$ on $X^\circ\setminus\Delta$. We claim that the metric regular part of $Y$ coincides with the open set $X^\circ\setminus\Delta$, which in turn implies that $Y$ is isometric to $\hat X_\epsilon$ by convexity of regular part \cite{Deng2021}. To prove the claim, we argue by contradiction. Suppose there is a metric regular point $p\in Y$ but not in $X^\circ\setminus \Delta$. We pick a sequence of points $p_i\in \hat X_{\epsilon,\delta_i}$ converging to $p$. By the lower semi continuity of volume density in the Gromov-Hausdorff topology, $p_i$ (up to subsequence) is not contained in the singular set of $X$. So we may assume that all points $p_i$ are contained in $\Delta$. If all points are contained in the regular part of $\Delta$, by the regularity theory of Monge-Ampere equation with conic singularity \cite{FGS25}, the volume density of $p$ will be no larger than the cone angle depending on the coefficient of $\Delta$, in particular, $p$ is not metric regular. If a subsequence of points $p_i$ are contained in the singular part of $\Delta$, then by the lower semi continuity of volume density in the Gromov-Hausdorff topology, again $p$ is not metric regular. This finishes the proof of claim.

At last we show that the metric completion $\hat X_\epsilon$ (or the Gromov-Hausdorff limit $Y$) is homeomorphic to $X$. We follow the argument of 
\cite[Section 6]{FGS25} and the only difference is that here we also need to deal with the boundary divisor $\Delta$. We sketch the proof. If there are two different points $p_1,p_2$ such that $\iota(p_1)=\iota(p_2)=q\in X$. By the Schwarz Lemma \cite{CCHSTT25}, there is Lipchitz map $\iota$ from $\hat X_\epsilon\rightarrow X$, so there is strongly pseudoconvex neighbourhood $ U$ of $q$ such that $\iota^{-1}{U}$ will contain a sufficiently small metric ball near $p_1,p_2$. If $q$ is a complex analytic singular point of $X$, then by \cite[Corollary 6.2]{FGS25} (or Appendix \ref{appI}), $\iota(p_1)\neq \iota(p_2)$,  contradiction. The key is that, in this case, the tangent cone at $p_1,p_2$ will capacity zero by \cite[Lemma 6.7]{FGS25} and Partial $C^0$ estimate applies \cite[Corollary 6.1]{FGS25}.  Now if $q$ is a complex analytic regular point of $X$, then $\omega_{\epsilon,\delta}$ is a sequence of smooth metric with Ricci curvature bounded below (independent of $\delta$) near $p_1$ and $p_2$, then by the partial $C^0$ estimate of \cite{LS1} (applied locally on a strongly pseudoconvex domain, c.f.\cite[Lemma 5.9]{FGS25} and Appendix \ref{appI}), in particular, $\iota(p_1)\neq \iota(p_2)$,  contradiction.

Finally by the Ricci lower bound on   $X^\circ \setminus \Delta$ \eqref{eqn:bbb},  we can improve the global Ricci lower bound to $-\epsilon$ by Honda's trick and so $ \hat X_\epsilon \in \RCD(-\epsilon, 2n)$
\end{proof}

Immediately, we can improve Proposition \ref{Prop:antinefpair} to the log canonical pair via a perturbation argument.
\begin{lemma} Let $(X,\Delta)$ be a  three dimensional log canonical pair and suppose that there is an $\R$-divisor $\Delta_0\ge0$ such that $(X,\Delta_0)$ is klt. If $-(K_X+\Delta)$ is nef, then for a fix K\"ahler class $[\theta]$ and $\epsilon>0$ sufficiently small, there is a K\"ahler current $\omega_\epsilon\in[\theta]$ and moreover $\omega_\epsilon$ induces a K\"ahler $\RCD(-2\epsilon,2n)$-space as in \eqref{eqn:rcd}
\end{lemma}

\begin{proof}
For any $0<\epsilon\ll1$ fixed, we have the following identity
\[
-(K_X+(1-\epsilon^2)\Delta+\epsilon^2\Delta_0)=-(K_X+\Delta)+(\epsilon^2\Delta-\epsilon^2\Delta_0+\epsilon\theta)-\epsilon\theta.
\]
So for $\epsilon$ sufficiently small, $-(K_X+(1-\epsilon^2)\Delta)+\epsilon\theta$ is a K{\"a}hler class. Moreover, $(X,(1-\epsilon^2)\Delta+\epsilon^2\Delta_0)$ is a klt pair. So by Proposition \ref{Prop:antinefpair}, we can construct a K\"ahler RCD(-2$\epsilon,2n$) space in a fixed K\"ahler class $[\theta]$.
\end{proof}

Now we proceed to prove that  $\pi_1(X)$ is virtually nilpotent under the assumption of Theorem \ref{thm:main1}, Corollary \ref{cor:main1}.
 
\subsection{Margulis Lemma and virtual nilpotency of \texorpdfstring{$\pi_1(X)$}{text}} 
In this subsection, we finish the proof of nilpotency of $\pi_1(X)$ by using Margulis Lemma. Let us first record the RCD version of Margulis Lemma of \cite{DSZZ}, which in turn is based on the idea of (\cite{CC96,KW}).
\begin{theorem}\label{DSZZ}{\cite{DSZZ}}
\label{magu} There exist $\varepsilon(n) > 0$ and $M(n) \in \mathbb{Z}^+$ such that if $(X, d, \mu)$ is a compact $m$ dimensional $RCD^*(-1,n)$ and $p\in X$, the image of the map
\begin{equation}\label{mor}\pi_1(B(p, \varepsilon(n))) \rightarrow \pi_1(X, p)\end{equation}
induced by the inclusion contains a subgroup of index $\leq M(n)$ that admits a nilpotent basis of
length $\leq n$.

\end{theorem} 
To apply the Margulis Lemma for the sequence of Ricci almost non negative RCD space $(X_\epsilon,d_\epsilon,\nu_\epsilon)$ constructed in previous subsection, it will be helpful to show that the morphism \eqref{mor} is surjective for $\epsilon$ sufficiently small. In general, it is hard to control the growth of  the diameters of $\omega_\epsilon$, the following trick from \cite{DPS93} will be useful, which says that we are able to control the length of suitable representatives of generators of $\pi_1(X)$ uniformly in $\epsilon$.  The key point is that $\omega_\epsilon$ lie in the same class.
 \begin{lemma}\label{dps} Fix a set of generators $\{\gamma_1,\gamma_2,...\gamma_n\}$ of $\pi_1(X)$. We may pick a point $p_\epsilon\in\tilde X$, where $\tilde X$ is the universal cover of $X$, such that there is a path connecting $p_\epsilon$ and $\gamma_j(p_\epsilon),j=1,2...n,$ contained in the regular part of $\tilde X$, whose $\omega_\epsilon$ length is bounded by a constant $D$ independent of $\epsilon$.
\end{lemma}

\begin{proof}
Fix an open set $U\subset \tilde X$ such that $U$ contains a fundamental domain $E$ and $\Vol_\omega(\gamma_j(U)\cap U)>\Vol_\omega(E)-\delta$. Since $\Vol_\omega(X_{sing})=0$, we further shrink $U$ such that $U$ lies in the regular part of $\tilde X$ with the following still holds,
\begin{equation}
\Vol_\omega(\gamma_j(U)\cap U)>\Vol_\omega(E)-\delta.\end{equation}
Now we pick a base point $p\in U$ and we apply Proposition \ref{DPS} to the open sets $U_1=U_2=U$, then we have that any two points $x_1\in U_{1,\delta,\epsilon}, x_2\in U_{2,\delta,\epsilon} $ can be joined by a path $\gamma$ (lies in the smooth part of $\tilde X$) of
length $\leq C\delta^{-1/2}$ with respect to $\omega_\epsilon$ and moreover, we still have
\begin{equation}
\Vol_\omega(\gamma_j(U_{\delta,\epsilon})\cap U_{\delta,\epsilon})>\Vol_\omega(E)-2\delta.\end{equation}
In particular $\gamma_j(U_{\delta,\epsilon})\cap U_{\delta,\epsilon}$ is nonempty, so we may fix a point $p_j\in\gamma_j(U_{\delta,\epsilon})\cap U_{\delta,\epsilon}$. To conclude, we may connect $p$ and $\gamma_j(p)$ by two pieces: one connecting $p$ and $p_j$ and one connecting $p_j=\gamma_j(q_j)$ and $\gamma_j(p)$ (Note $\gamma_j$ is an isometry of $\tilde X$). By Proposition \ref{DPS}, the total length is bounded by $2C\delta^{-1/2}$, which is independent of $\epsilon$.
\end{proof}

 Now we can easily prove the virtual nilpotency property of $\pi_1(X)$ based on the idea of P{\u a}un. 
 \begin{lemma} \label{lemma:an}Under the assumption of Theorem \ref{thm:main1}, Corollary \ref{cor:main1}, $\pi_1(X)$ is virtually nilpotent.

\end{lemma}
\begin{proof}Let $\epsilon(n)$ be the constant in  Lemma \ref{dps}. By choosing $\epsilon$ sufficiently small such that $\epsilon D^2\leq \epsilon(n)$ (D is the constant in Lemma \ref{dps}), it follows from the Lemma \ref{dps}  that $\pi_1(B_d(p, \varepsilon(n))) \rightarrow \pi_1(X, p)$ is surjective. Hence $\pi_1(X)$ is virtually nilpotent by Theorem \ref{DSZZ}. 
\end{proof}
 
\begin{remark}\label{add} It is clear from the proof that if conjecture \ref{conrcd} holds and the variety $X$ satisfies the assumption of Theorem \ref{thm:main0}, we may also construct a sequence of Ricci almost non negative RCD spaces which are homeomorphic to $X$ and hence $\pi_1(X)$ is almost nilpotent by Margulis Lemma.
\end{remark}

%%%%%%%%%%%%%%%%%%%%%%%%%

\section{From virtual nilpotence to almost abelianity of \texorpdfstring{$\pi_1(X)$}{text}}
\label{sec:abelianity}
%In this section we give the proof of {\hyperref[thm:main1]{Theorem \ref*{thm:main2}}, preciesly, we show that for normal projective variety $X$ with klt singularities and nef anticanonical divisor, its fundamental group $\pi_1(X)$ is almost Abelian. Indeed, the theorem can be deduced immediately from {\hyperref[thm:main1]{Theorem \ref*{thm:main1}} 
The main object of this section is to deduce the almost abelianity of $\pi_1(X)$ from the virtual nilpotency result Lemma \ref{lemma:an} established in previous sections. Using a result of Camapana (see Appendix \ref{app:Campana} for a general account), the main task is to show that the Albanese map is surjective. To this end, we need the following the positivity theorem for twisted relative canonical divisors analogous to \cite[Proposition 3.1]{MWWZ25}.
\begin{proposition}
\label{prop:positivity}
Let $f: X \to Y$ be a projective morphism between compact K{\"a}hler manifolds. Let $A$ be an $f$-relatively ample line bundle on $X$ and $N$ a $\Q$-line bundle on $X$ such that $c_1(N)$ is pseudoeffective and asymptotically klt. 
Assume that $f_*\mathscr{O}_X(mK_{X/Y} + A + mN )$ is a non-zero sheaf for $m\in\Z_{>0}$ sufficiently large and divisible. Then the $\Q$-line bundle $K_{X/Y}+N-\Ram(f)$ is pseudo-effective, where $\Ram(f)$ denotes the ramification divisor of $f$, i.e. 
\[
\Ram(f):=\sum_i (e_i-1)W_i\,,
\]
where $W_i$ runs over all the prime divisors on $X$ such that $f(W_i)$ is of codimension $1$ on $Y$ and $e_i$ is the multiplicity of $f^\ast f(W_i)$ along $W_i$.
\end{proposition}

\begin{proof}
The proof is quite similar to that of \cite[Proposition 3.1(a)]{MWWZ25}. Fix K{\"a}hler forms $\omega_X$ and $\omega_Y$ on $X$ and $Y$ respectively. Since $N$ is pseudoeffective and asymptotically klt, for every $\epsilon>0$ we can construct a singular Hermitian metric $h_\epsilon$ on $N$ such that  $\sqrt{-1}\Theta_{h_\epsilon}(N)\geq -\epsilon\omega_X$ and $\mathscr{J}(h_\epsilon)=\mathscr{O}_X$ by Remark \ref{rmk:asym_klt}(c), which implies that $\mathscr{J}(h_\epsilon|_{X_y})=\mathscr{O}_{X_y}$ for $y\in Y$ a.e. by \cite[Remark 3.29]{Pau18}. Since $A$ is $f$-relatively ample, we have a $\mathscr{C}^{\infty}$ Hermitian metric $h_A$ on $A$ such that $\sqrt{-1}\Theta_{h_A}(A)+f^\ast\omega_Y\geq\omega_X$. For every $m\in\Z_{>0}$, we choose $\epsilon_m\le 1/m$ and consider the singular metric $h_m:=h_{\epsilon_m}^m\cdot h_A$ on $A+mN$, then we have 
\begin{equation}
\label{h_m_curv}
\sqrt{-1}\Theta_{h_m}(A+mN)\geq -m\epsilon_m\omega_X+\omega_X-f^\ast\omega_Y\geq -f^\ast\omega_Y.
\end{equation}
Moreover, since $h_A$ is a $\mathscr{C}^{\infty}$-metric, we have
\begin{equation}
\mathscr{J}(h_m^{1/m})=\mathscr{J}(h_{\epsilon_m}\cdot h_A^{1/m})=\mathscr{J}(h_{\epsilon_m})=\mathscr{O}_X.
\end{equation}
Then by H{\"o}lder's inequality and our assumption, for $y\in Y$ general, there is $0\neq u\in H^0(X_y, \mathscr{O}_{X_y}(mK_{X_y}+A+mN))$ satisfying
\[
\int_{X_y}|u|^{\frac{2}{m}}e^{-\frac{\phi_m}{m}}<+\infty,
\]
where $\phi_m$ denotes the local weight of $h_m$. Then we can construct the relative $m$-Bergman kernel metric $h_{X/Y,h_m}^{(m)}$ on $mK_{X/Y}+A+mN$ (see \cite[\S A.2]{BP10}, \cite[\S 4.2]{PT18} and \cite[\S 3.2]{Cao17}) so that
\begin{equation}
\label{m-Bergman_curv}
\sqrt{-1}\Theta_{h_{X/Y,h_m}^{(m)}}(mK_{X/Y}+A+mN)-m[\Ram(f)]\geq -f^\ast\omega_Y
\end{equation}
by \cite[Theorem 2.3]{Wan21} (see also \cite[Remark 2.5]{CP17}, \cite[Theorem 1.2]{CP25} and \cite[Theorem 5.1]{MWWZ25}). Here the subtle point is that we cannot directly apply \cite[Theorem 0.1]{BP10} (\cite[Theorem 4.2.7]{PT18}, \cite[Theorem 3.5]{Cao17}) since the curvature current of $h_m$ is not necessarily semipositive, instead we can only apply them locally over $Y$ in the following way: locally we write $\omega_Y=dd^c\varphi$ for some smooth strictly psh function $\varphi$, then from \eqref{h_m_curv} we have $\sqrt{-1}\Theta_{he^{-f^\ast\varphi}}(A+mN)\geq0$ and we can apply \cite[Theorem 0.1]{BP10} to get the relative $m$-Bergman kernel metric $h_{X/Y,h_me^{-f^\ast\varphi}}^{(m)}$ on $mK_{X/Y}+A+mN$ with semipositive curvature current (these can only be defined locally and cannot glue together!), by construction, locally we have 
\[
h_{X/Y,h_me^{-f^\ast\varphi}}^{(m)}=h_{X/Y,h_m}^{(m)}\cdot e^{-f^\ast\varphi},
\]
hence we get the desired curvature bound \eqref{m-Bergman_curv}. Now let $m\to+\infty$, \eqref{m-Bergman_curv} ensures the existence of a closed positive current in $c_1(K_{X/Y}+N-\Ram(f))$, in other word $K_{X/Y}+N-\Ram(f)$ is pseudoeffective and the proposition is thus proved.
\end{proof}

\begin{remark}
\label{rmk:positivity}
$($a$)$ The main difference between Proposition \ref{prop:positivity} and \cite[Proposition 3.1]{MWWZ25} is that when $N$ is nef, we can construct a series of $\mathscr{C}^\infty$-metrics $h_\epsilon$ on $N$ of which the curvature form have arbitrarily small negative part, while in our case the metrics $h_\epsilon$ are in general singular; although the singularities of $h_\epsilon$ are mild (having trivial multiplier ideal), we cannot guarantee that $h_\epsilon\cdot h_L$ still have trivial multiplier ideal for some $\epsilon>0$, hence we lose the flexibility to twist the additional line bundle $L$; for the same reason, we should assume $A$ to be relatively ample, instead of relatively big.      

$($b$)$ In the statement of Proposition \ref{prop:positivity}, we can slightly weaken the condition on $N$: it suffices to assume that $N$ is pseudoeffective and that the cosupport of $\mathscr{J}_+(c_1(N))$ (see Remark \ref{rmk:asym_klt}(a)) does not dominate $Y$. See also \cite[Theorem 2.1]{Wan22}. 
\end{remark}

\subsection{Surjectivity of the Albanese map}
\label{ss:Alb}

The main goal of this subsection is to prove Theorem \ref{thm:Alb} as well as the surjectivity of Albanese maps under the assumption of Corollary \ref{cor:main1}. 

%As for Theorem \ref{cor:main1}, surjectivity of the Albanese map has recently been established by \cite[Corollary 3.13]{MWWZ25} (see Theorem \ref{thm:Alb}). 
%\textcolor{red}{confusing here}
\begin{theorem}
\label{thm:psef-star-Alb}
Let $X$ be a compact K{\"a}hler manifold. Suppose that $-K_X$ is pseudoeffective and asymptotically klt. Then the Albanese map $\alb_X: X\to\Alb_X$ of $X$ is a fibration.
\end{theorem}

Theorem \ref{thm:psef-star-Alb} can be regarded as the generalization of  \cite[Proposition 2.7.1]{DPS01}: the surjectivity of $\alb_X$ is established if $-K_X$ is pseudoeffective and $\mathscr{J}(c_1(-K_X))=\mathscr{O}_X$ (see Definition \ref{def:asym_klt}). In fact they can prove that $\alb_X$ is a holomorphic submersion (smooth morphism).

%It follows immediately from the theorem above and Lemma \ref{1-2} that the Albanese map of a compact K{\"a}hler mainfold $X$ is a fibration if there is an effective $\R$-divisor $\Delta$ on $X$ such that $(X,\Delta)$ is log canonical and $-(K_X+\Delta)$ is nef. 
%%Indeed, by writing $-K_X=-(K_X+\Delta)+\Delta$ we find that $-K_X$ is pseudoeffective and asymptotically klt; to see this, just notice the following fact: let $\theta$ be a K{\"a}hler form on $X$ and put $\eta:=c_1(-(K_X+\Delta))$, then for any $\epsilon>0$ the class $\eta+\epsilon\cdot\{\theta\}$ is K{\"a}hler thus so is $\eta+\epsilon\cdot\{\theta\}+\delta(\epsilon)\cdot\{\Delta\}$ for $\delta(\epsilon)>0$ sufficiently small since the K{\"a}hler cone is open. 
%In summary, we have the result below:

%\begin{corollary}
%\label{cor:nef-lc-Alb}
%Let $X$ be a compact K{\"a}hler manifold. Suppose that there is an effective $\R$-divisor $\Delta$ on $X$ such that $(X,\Delta)$ is log canonical and $-(K_X+\Delta)$ is nef. Then the Albanese map $\alb_X: X\to\Alb_X$ of $X$ is a fibration.
%\end{corollary}

%If $\Delta$ has SNC support, then the surjectivity of $\alb_X$ has been obtained in \cite[Theorem 1.3]{FHZ25} by establishing a positivity theorem for the variation of fibrewise K{\"a}hler-Einstein metrics. And if $\Delta$ is a $\Q$-divisor, we can generalize this corollary to singular case.

The proofs of Theorems \ref{thm:Alb} and \ref{thm:psef-star-Alb} are essentially the same as that of \cite[Corollary 3.13]{MWWZ25}, as long as we establish Proposition \ref{prop:Zhang} below, which is a generalization of \cite[Main Theorem]{Zha05} and \cite[Proposition 3.12(a)]{MWWZ25}:

\begin{proposition}
\label{prop:Zhang}
Let $X$ be a compact normal complex variety in Fujiki class $\mathscr{C}$ and assume that $X$ satisfies the condition of Theorem \ref{thm:Alb} or Theorem \ref{thm:psef-star-Alb}. Let $\psi: X\dashrightarrow Y$ be a meromorphic mapping with $Y$ being a compact K{\"a}hler manifold and let $\phi: M\to Y$ be an elimination of indeterminacies of $\psi$ with $\pi:M\to X$ being the natural bimeromorphic morphism. Assume moreover that $\phi$ is projective and $Y$ is not uniruled. Then we have
\begin{itemize}
\item[(a)] $K_Y$ is $\Q$-linearly equivalent to an effective divisor $D_Y$, such that $\phi^\ast D_Y$ is $\pi$-exceptional, and $[D_Y]$ is the unique closed semipositive $(1,1)$-current in $c_1(K_Y)$. In particular, $\nu(X)=\kappa(Y)=0$.  
\item[(b)] $\Delta$ is horizontal with respect to $\psi$.
\item[(c)] $\Ram(\phi)$ is $\pi$-exceptional.
\end{itemize}
\end{proposition}

Before giving the proof of Proposition \ref{prop:Zhang}, let us briefly explain how to deduce Theorems \ref{thm:Alb} and \ref{thm:psef-star-Alb} from Proposition \ref{prop:Zhang}, for the convenience of the readers: 
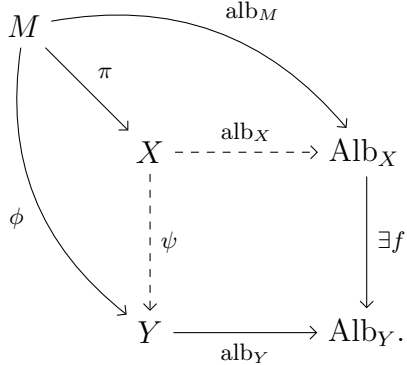
\begin{center}
\begin{tikzpicture}[scale=2.4]
\node (A) at (0,0) {$Y$};
\node (B) at (-0.7,1.7) {$M$};
\node (C) at (0,1) {$X$};
\node (A1) at (1.2,0) {$\Alb_Y$.};
\node (C1) at (1.2,1) {$\Alb_X$\,};
\path[->,font=\scriptsize,>=angle 90]
(B) edge[bend right] node[below left]{$\phi$} (A)
(B) edge node[above right]{$\pi$} (C)
(B) edge[bend left] node[above right]{$\alb_M$} (C1)
(A) edge node[below]{$\alb_Y$} (A1)
(C1) edge node[right]{$\exists f$} (A1);
\path[dashed,->,font=\scriptsize,>=angle 90]
(C) edge node[right]{$\psi$} (A)
(C) edge node[above]{$\alb_X$} (C1);
\end{tikzpicture}
\end{center}
Let $X$ be a compact normal complex variety in Fujiki class $\mathscr{C}$ satisfying the condition of Theorem \ref{thm:Alb} or \ref{thm:psef-star-Alb}. Let $\psi: X\dashrightarrow Y$ be the MRC fibration of $X$, and let $\phi$ and $\pi$ as in Proposition \ref{prop:Zhang}. Since $X$ is in Fujiki class $\mathscr{C}$, we may arrange $M$ and $Y$ to be compact K{\"a}hler manifolds, and let $\alb_X: X\dashrightarrow\Alb_X$ (resp. $\alb_Y: Y\to\Alb_Y$) be the Albanese map of $X$ (resp. $Y$). Then $\alb_X\circ\pi$ coincides with the Albanese map $\alb_M$ of $M$, which yields a morphism of complex tori $f:\Alb_X\to \Alb_Y$ from the universal property of the Albanese maps, and we have the commutative diagram above. From \cite[Theorem 1.2]{CH24} we know that $\phi$ is a projective morphism, then we can apply Proposition \ref{prop:Zhang}(a) and conclude that $\kappa(Y)=0$, which in turn implies that $\alb_Y$ is a fibration by \cite[Theorem C]{Wan21} and thus $f$ is also a fibration. On the other hand, from the property of MRC fibrations, we have $H^1(M,\mathscr{O}_M)\simeq H^1(Y,\mathscr{O}_Y)$, hence $\Alb_X$ and $\Alb_Y$ have the same dimension. Since $f$ is surjective, it must be a isogeny (see \cite[\S 2.3-2.4, pp.7-8]{Deb99}); but $f$ is a fibration, it is a fortiori an isomorphism. Therefore $\alb_M$ is a fibration. In particular, if $X$ is smooth, $\alb_X$ is also an everywhere defined fibration.

Now let us turn to the proof of Proposition \ref{prop:Zhang}. In the remaining part of this subsection, all the pullbacks are taken in the sense of $\Q$-line bundles (or $\R$-line bundles).
\begin{proof}[Proof of {Proposition \ref{prop:Zhang}}]
With Proposition \ref{prop:positivity} being established, the proof of the proposition is almost the same as \cite[Proposition 3.13(a)]{MWWZ25} and the key ingredients is the significant progress in K{\"a}hler geometry recently made by Ou \cite{Ou25} (see also \cite{CP25}). In the sequel we will treat the case of Theorem \ref{thm:Alb} and that of Theorem \ref{thm:psef-star-Alb} uniformly, to this end let us first fix some notations: 
%%%%we can write $\phi^\ast K_Y\sim\pi^\ast L+D$ for some line bundle $L$ on $X$ and some $\pi$-exceptional divisor $D$, where $L$ denotes the reflexive hull of the direct image sheaf $\pi_\ast\mathscr{O}(\phi^\ast K_Y)$ (cf. \cite[Lemma 3.4]{MWWZ25}, in our case, since $X$ is smooth, we can argue ad-hocly); moreover, since $X$ is smooth, we can write $K_M\sim \pi^\ast K_X+F$ for some effective $\pi$-exceptional divisor $F$.
Up to further blowing up $M$, we may assume that there is an effective $\Q$-divisor $F_+$ on $M$ such that $N:=-K_M+F_+$ is pseudoeffective and asymptotically klt. Indeed, in the case of Theorem \ref{thm:psef-star-Alb}, we can write $K_M=\pi^\ast K_X+F_+$ for $F_+\ge0$ $\pi$-exceptional $\Z$-divisor, then $N:=-K_M+F_+=-\pi^\ast K_X$ is pseudoeffective and asymptotically klt; while for the case of Theorem \ref{thm:Alb} we may assume that $\pi$ is a log resolution of $(X,\Delta)$ and we can write 
\[
K_M+\Delta_M\sim_\R \pi^\ast(K_X+\Delta)+F_+,
\]
where $\Delta_M=\pi^{-1}_\ast\Delta+F_-$, $F_+$ and $F_-$ are effective $\pi$-exceptional $\R$-divisors with no common components, and they all have SNC supports, then by log canonicity we see that $\Delta_M\le 1$ (i.e. the coefficients of the components in $\Delta_M$ are all $\le 1$), hence $N:=-K_M+F_+\sim_{\Q} -\pi^\ast(K_X+\Delta)+\Delta_M$ is pseudoeffective and asymptotically klt (see the proof of Lemma \ref{1-2}), furthermore, we can perturb $F_+$ to a $\Q$-divisor $F'_+$ such that $F'_+-F_+$ is sufficiently small and effective with support contained in $\Supp(F_+)$, then, by replacing $F_+$ by $F'_+$ we may arrange $N=-K_M+F_+$ to be a $\Q$-line bundle and $N$ remains pseudoeffective and asymptotically klt since $\Delta_M$ and $F_+$ have no common components. Then we show the item (a) in 4 steps as in the proof of \cite[Proposition 3.13]{MWWZ25}: 
\begin{enumerate}
\item Applying Proposition \ref{prop:positivity} to $\phi: M\to Y$ and $N=-K_M+F_+$, we obtain that $-\phi^\ast K_Y+F_+$ is pseudoeffective. Let $\tilde T$ be a closed positive $(1,1)$-current in $c_1(-\phi^\ast K_Y+F_+)$, then locally $\tilde T=dd^c\tilde\varphi$ for some (local) psh function $\tilde\varphi$ on $M$, by Grauert-Remmert's extension theorem \cite[Satz 4, p.181]{GR56} $\tilde\varphi|_{X\backslash\pi(E)}$ extends to a (local) psh function $\varphi$ on $X$ (see \cite[D{\'e}finition 1.5]{Dem85}), thus we get a closed positive $(1,1)$-current $T$ with local potential on $X$ (see \cite[Definition 4.6.2, p.430]{BEG13}), which defines a Bott-Chern class $\alpha=\{T\}\in H^{1,1}_{\BottChern}(X)$ such that $\alpha=\pi_\ast c_1(-\phi^\ast K_Y+F_+)=-\pi_\ast\phi^\ast c_1(K_Y)$. Here our argument avoids the use of \cite[Lemma 3.4]{MWWZ25} which requires the strong $\Q$-factoriality of $X$.  
\item Since $Y$ is not uniruled, by \cite[Theorem 1.1]{Ou25} $K_Y$ is pseudoeffective and thus so is $-\alpha$ by the same argument as above. Combine this with (1) we see that $\alpha=0$. 
\item Take a closed positive $(1,1)$-current $T_0\in c_1(K_Y)$, then $\pi_\ast\phi^\ast T_0=0$, thus the support of $T_0$ is contained in $\Exc(\pi)$, and by the second theorem of support \cite[III.2.14, p.143]{agbook} we see that $\phi^\ast T_0=[D]$ for some effective $\pi$-exceptional $\R$-divisor $D$; moreover, since $[D\cap M_y]=\phi^\ast T_0|_{M_y}=0$ for general $y\in Y$, $D$ is vertical and again by the theorem of support we have $T_0=[D_Y]$ for some $D_Y\geq 0$ and $D=\phi^\ast D_Y$. Moreover, $T_0=[D_Y]$ is unique in $c_1(K_Y)$ by the negativity lemma \cite[Lemma 1.3]{Wan21}.
\item From (3) we see that $K_Y\equiv D_Y\geq 0$, by \cite[Corollary 5.9(b)]{Wan21} we have $K_Y\sim_{\Q} D_Y$. Moreover, since $D=\phi^\ast D_Y$ is $\pi$-exceptional, we have $\nu(Y)=\kappa(Y)=0$ (see \cite[Definition (18.13), p.198]{Dem10}). And (a) is thus established.
\end{enumerate}
The items (b) and (c) are direct consequence of (a): by applying Proposition \ref{prop:positivity} to $N=-(K_M+\Delta_M^v)+F_+\sim_{\Q}-\pi^\ast(K_X+\Delta)+\Delta^h_M$ we have that $-\phi^\ast K_Y-\Delta_M^v-\Ram(\phi)+F_+$ is pseudoeffective. By pushing forward via $\pi_\ast$ we obtain that $-\pi_\ast(\Delta_M^v+\Ram(\phi))$ is pseudoeffective, implying that $\pi_\ast\Delta_M^v=\pi_\ast\Ram(\phi)=0$ since both are effective, and (b) and (c) are thus established. 
\end{proof}

\subsection{On a question of Demailly-Peternell-Schneider for K{\"a}hler varieties}
\label{ss:DPS}
Before entering into the final step of the proof of our main theorems, let us observe that the positivity result of Proposition \ref{prop:positivity} leads to the following generalization of Chen-Zhang's theorem \cite[Main Theorem]{CZ13} and Deng's theorem \cite[Theorem B]{YDeng21}, thereby answering affirmatively a classical question of Demailly-Peternell-Schneider \cite[Problem 4.13.a]{DPS01} (see also \cite[Question 3.4]{Pet12}) in the context of compact complex varieties in Fujiki class $\mathscr{C}$ (note however that we should assume the morphism $f$ to be projective). Before stating these results, let us observe the following useful lemma, which can be shown by the same argument as in Step (1) of the proof of Proposition \ref{prop:Zhang} or directly by \cite[Proposition  4.6.3, pp.230-231]{BEG13}.

\begin{lemma}
\label{lemma:bimero-psef}
Let $X$ be a compact normal complex variety and let $g: X'\to X$ be proper modification with smooth $X'$. For $\alpha\in H^{1,1}_{\BottChern}(X)$, $\alpha$ is pseudoeffective if and only if $g^\ast\alpha+c_1(E)\in H^{1,1}_{\BottChern}(X')$ is pseudoeffective for some (not necessarily effective) $g$-exceptional divisor $E$.   
\end{lemma}

\begin{theorem}
\label{thm:Chen-Zhang-Deng}
Let $f:X\to Y$ be projective morphism between compact normal complex varieties, such that $X$ is in Fujiki class $\mathscr{C}$ and admits an effective $\R$-divisor $\Delta$ such that $(X,\Delta)$ is a log canonical pair and $-(K_X+\Delta)$ is pseudoeffective and that $Y$ is $\Q$-Gorenstein. Assume that there is a desingularization $\pi:X'\to X$ and a $\Q$-line bundle $L_Y$ on $Y$ such that the non-nef locus of $c_1(-\pi^\ast(K_X+\Delta)-\pi^\ast f^\ast L_Y)$ does not dominate $Y$. 
Then $-K_Y-L_Y$ is pseudoeffective. 

In particular, if $L_Y=0$ and $-(K_X+\Delta)$ is nef, then $-K_Y$ is pseudoeffective. 
\end{theorem}
\begin{proof}
By \cite[Th{\'e}or{\`e}me 3]{Var86} $Y$ is also in Fujiki class $\mathscr{C}$ and thus we can take a desingularization $\pi_Y:Y'\to Y$ such that $Y'$ is a compact K{\"a}hler manifold. By Lemma \ref{lemma:pullback-nnef} we may freely replace $X'$ by a further blow up, so that $X'$ is a compact K{\"a}hler manifold, $\pi:X'\to X$ is a projective modification and a log resolution of $(X,\Delta)$ and $f': X'\to Y'$ is an elimination of indeterminacies of $X\dashrightarrow Y'$. Pick a K{\"a}hler form $\omega$ on $X'$. Then we can write 
\[
K_{X'}+\Delta'\sim_{\R} \pi^\ast(K_X+\Delta)+F_+\,,
\]
where $\Delta'=\pi_\ast^{-1}\Delta+F_-$, $F_+$ and $F_-$ are effective $\pi$-exceptional $\R$-divisors with no common components; since $(X,\Delta)$ is log canonical, $\Delta'\le1$, hence for every $\epsilon>0$ there is a current $T_{\Delta',\epsilon}\in c_1(\Delta')$ such that $T_{\Delta',\epsilon}\ge-\frac{\epsilon}{2}\omega$ and $\mathscr{J}(\varphi_{\Delta',\epsilon})=\mathscr{O}_{X'}$ where $\varphi_{\Delta',\epsilon}$ is a local potential of $T_{\Delta',\epsilon}$. By assumption, the non-nef locus of $c_1(-\pi^\ast(K_X+\Delta)-\pi^\ast f^\ast L_Y)$ does not dominate $Y$, then for every $\epsilon>0$, there is a closed positive $(1,1)$-current $T^\circ_\epsilon\in c_1(-\pi^\ast(K_X+\Delta)-\pi^\ast f^\ast L_Y)$ such that $T^\circ_\epsilon\ge -\frac{\epsilon}{4}\omega$ and $\nu(T^\circ_\epsilon, x)=0$ for every $x\in X'_y$ for a general $y\in Y$, then by Demailly's regularization (see \cite[Theorem 2.1(ii)]{Bou04}), there is a a closed positive $(1,1)$-current $T_\epsilon\in c_1(-\pi^\ast(K_X+\Delta)-\pi^\ast f^\ast L_Y)$ with analytic singularities such that $T_\epsilon\ge-\frac{\epsilon}{2}\omega$ and $\nu(T_\epsilon,x)=0$ for every $x\in X'_y$, in particular this implies that $T_\epsilon|_{X'_y}$ is a smooth $(1,1)$-form. Then consider the current $T_{N,\epsilon}=T_{\Delta',\epsilon}+T_\epsilon\in c_1(N)$ where $N:=-K_{X'}+F_+-\pi^\ast f^\ast L_Y=-\pi^\ast(K_X+\Delta-f^\ast L_Y)+\Delta'$, we have $T_{N,\epsilon}\ge-\epsilon\omega$ and $\mathscr{J}(\varphi_N|_{X'_y})=\mathscr{O}_{X'_y}$ for general $y\in Y$, hence the cosupport of $\mathscr{J}_+(c_1(N))$ does not dominate $Y$. Moreover, as in the proof of Proposition \ref{prop:Zhang}, we can perturb $F_+$ to render it a $\Q$-divisor, so that $N=-K_{X'}+F_+-\pi^\ast f^\ast L_Y$ is $\Q$-line bundle and remains pseudoeffective and the cosupport of $\mathscr{J}_+(c_1(N))$ does not dominate $Y$. Then Proposition \ref{prop:positivity} and Remark \ref{rmk:positivity}(b) implies that $-(f')^\ast K_{Y'}-\pi^\ast f^\ast L_Y+F_+$ is pseudoeffective, hence so is $-f^\ast(K_Y+L_Y)$ by Lemma \ref{lemma:bimero-psef} since $-(f')^\ast K_{Y'}-\pi^\ast f^\ast L_Y+F_++ \pi^\ast f^\ast(K_Y+L_Y)$ is a $\pi$-exceptional divisor, thus $-\pi^\ast f^\ast(K_Y+L_Y)=-(f')^\ast\pi_Y^\ast(K_Y+L_Y)$ is pseudoeffective. Finally, by Lemma \ref{lemma:Kahler-cohomology} we find that $-\pi_Y^\ast(K_Y+L_Y)$ is pseudoeffective and so is $-K_Y-L_Y$. The theorem is thus established. 
\end{proof}

\subsection{Specialness of compact K{\"a}hler varieties with nef anticanonical divisor}
\label{ss:special}
Let $X$ be a compact variety in the Fujiki class $\mathscr{C}$ that admits an effective $\R$-divisor $\Delta$ such that the pair $(X,\Delta)$ is log canonical and the anti-log canonical divisor $-(K_X+\Delta)$ is nef. 
As promised in \S \ref{ss:motivation}, we will prove in this section that $(X,\Delta)$ is a special orbifold in the sense of Campana \cite[Definition 2.41]{Cam04}. 
%The general reference for this subsection is \cite{Cam04}, we invite the reader to consult this article for all the related notions, like specialness \cite[Definition 2.1 \& Definition 2.41]{Cam04} and the core map. 
Before giving the proof of this result, we first briefly recall the related notions in \cite{Cam04,Cam11}. Let $f:X\to Y$ be fibration between compact complex varieties and let $(X,\Delta)$, the basic observation in \cite{Cam04} is that, in order to comprehend the geometry of this fibration, one should take into consideration, not only $(X,\Delta)$ and $Y$, but also an `orbifold base' divisor $B_{f,\Delta}$ on $Y$ that encodes the ramification of $f$ in codimension $1$, even if $\Delta=0$. From its definition \cite[\S 1.6.1]{Cam04} we find that $B_{f,\Delta}$ is most effective divisor $B$ such that $f^\ast B\leqslant \Ram(f)+\Delta^{v}+E$ for some exceptional divisor $E$, where $\Delta^h$ (resp. $\Delta^v$) denotes the horizontal part (resp. vertical) part of the divisor $\Delta$. See \cite[Lemma 7.2 \& Remark 7.3]{Wan21}. 

For a meromorphic fibration $f:(X,\Delta)\dashrightarrow Y$, the canonical dimension of $f$, denoted by $\kappa(f,\Delta)$, is defined to be the Kodaira dimension $\kappa(Y',B_{f',\Delta'})$ for a sufficiently high holomorphic model $f': (X',\Delta')\to Y'$ of $f$ (see \cite[Definition 2.40]{Cam04}). A meromorphic $f$ is called a {\itshape general type} fibration if $\dim Y>0$ and $\kappa(f,\Delta)=\dim Y$, and we say that $(X,\Delta)$ is an {\itshape special orbifold} if it does not admit any general type fibration (see \cite[Definition 2.41]{Cam04}); in particular, a variety $X$ is called {\itshape special} if $(X,0)$ is a special orbifold. Moreover, for an lc pair $(X,\Delta)$, among all the general type fibrations $(X,\Delta)\dashrightarrow Y$ (including the constant map), there is a unique one (up to bimeromorphic equivalence) that dominates all the others (in the bimeromorphic sense) by \cite[Th{\'e}or{\`e}me 10.1]{Cam11}, and it called the {\itshape core map} of $(X,\Delta)$; clearly, $(X,\Delta)$ is a special orbifold if its core map is a constant map. Now let us turn to the proof of: 

\begin{theorem}[K{\"a}hler version of {\cite[Theorem 10.3]{Cam16}}]
\label{thm:special}
Let $X$ be a compact variety in the Fujiki class $\mathscr{C}$ that admits an effective $\R$-divisor $\Delta$ such that the pair $(X,\Delta)$ is log canonical and the anti-log canonical divisor $-(K_X+\Delta)$ is nef. Then $X$ is special. Moreover, if $\Delta$ is a $\Q$-divisor, then $(X,\Delta)$ is a special orbifold. 
\end{theorem}
\begin{proof}
For the first assertion, consider the core map $c_X:X\dashrightarrow (C_X,B_c)$ of $X$ and the MRC fibration $\psi:X\dashrightarrow Y$. Since rationally connected varieties are special, by \cite[Theorem 3.7]{Cam04} we have a fibration $g: Y\dashrightarrow C_X$ such that $c_X=\psi\circ g$. Let $\phi: M\to Y$ be an elimination of indeterminacies of $\psi$, and up to passing to a higher model, we may assume that $g$ is a morphism and $c_X\circ\pi=\phi\circ g$ and we put $B_Y:=B_{\phi,\Delta_M}$. By Proposition \ref{prop:Zhang}(b)(c) and \cite[Lemma 7.2 \& Remark 7.3]{Wan21} we see that $\phi^\ast B_Y$ is $\pi$-exceptional. By \cite[Theorem 4.11]{Cam04} $\mathscr{F}:=g_\ast(m_0(K_{Y/C_X}+B_Y-g^\ast B_{g,B_Y})+E)$ is weakly positive for every sufficiently divisible $m_0\in\Z_{>0}$ and for some $g$-exceptional divisor $E$. In particular, for any ample line bundle $A$ on $C_X$ and for each $m\in\Z_{>0}$, there is $b\in\Z_{>0}$ such that we have a non-zero map
\[
\mathscr{O}_{C_X}\to\Sym^{bm}\mathscr{F}\otimes A^{\otimes b}\to g_\ast(bmm_0(K_{Y/C_X}+B_Y-g^\ast B_{g,B_Y})+bmE+bA),
\]
which is a fortiori an injection. Moreover, by Kodaira lemma we may assume that $m_0(K_{C_X}+B_{g,B_Y})\ge A$, so that We have
\begin{align*}
h^0(Y,bmm_0(K_Y+B_Y)+bmE) &\ge h^0(C_X,(bm-1)m_0(K_{C_X}+B_{g,B_Y}))\\
&\ge h^0(C_X,b(m-1)m_0(K_{C_X}+B_{c}))
\end{align*}
where the second inequality results from \cite[Proposition 1.30]{Cam04}. Since we may assume that $g$ is neat in the sense of \cite[Definition 1.2]{Cam04}, we have
\[
h^0(Y,bmm_0(K_Y+B_Y))\ge h^0(C_X,b(m-1)m_0(K_{C_X}+B_c)),
\]
By Proposition \ref{prop:Zhang}(a) we see that the LHS is upper bounded by a constant independent of $m$, but $(C_X,B_c)$ is of log general type, this implies that $\dim C_X=0$ and the first assertion is thus proved.

The second assertion is proved in \cite[Theorem 10.3]{Cam16} when $X$ is projective (although Campana states the result for smooth $X$, it is easy to observe that the argument also works if $X$ is singular and Moishezon; indeed, one can also give a more analytic proof based on Proposition \ref{prop:positivity} as explained below). Nevertheless, Campana's proof cannot be generalized to the K{\"a}hler case for the following 2 reasons:
\begin{itemize}
\item The positivity theorem for the twisted relative canonical bundle fails if the morphism is not projective. Of course, we can mimic the proof of the the first assertion by considering the sRC quotient (see \cite[Theorem 1.5]{Cam16}), but its construction is unclear if the variety is not projective.
\item Even we have a reasonable construction of the sRC quotient for compact K{\"a}hler manifolds, in order to conclude as in Campana's argument, we need to establish the Abundance conjecture for lc pairs in the numerical dimension $0$ case, which is unknown in the K{\"a}hler case (the projective case is established by Kawamata \cite[Theorem 1]{Kaw13} and \cite[Theorem 0.1]{CKP12}), see also \cite[Remark 5.8]{Wan21}.
\end{itemize}
In order to circumvent these difficulties, we consider the core map $c_{X,\Delta}: (X,\Delta)\dashrightarrow (C_{X,\Delta},B_{c,\Delta})$ of $(X,\Delta)$ and the MRC fibration $\psi:X\dashrightarrow Y$ of $X$. The general fibre $F$ of $\psi$ (note that MRC fibrations are almost holomorphic) is rationally connected, in particular, it is Moishezon, thus we can apply \cite[Theorem 10.3]{Cam16} (see Theorem \ref{thm:sRC} below) to obtain that $(F,\Delta|_F)$ is a special orbifold. Then, similar to the proof of the first assertion, we get a factorization $g:Y\dashrightarrow C_{X,\Delta}$, and the same argument as in the proof of the first assertion shows that $\dim C_{X,\Delta}=0$, hence the second assertion.  
\end{proof}

To end up this subsection, we briefly sketch the proof of the following bimeromorphic version of \cite[Theorem 10.3]{Cam16} for the convenience of the readers.
\begin{theorem}
\label{thm:sRC}
Let $X$ be a compact Moishezon variety that admits an effective $\Q$-divisor $\Delta$ such that the pair $(X,\Delta)$ is log canonical and the anti-log canonical divisor $-(K_X+\Delta)$ is nef, and let $f:(X,\Delta)\dashrightarrow Z$ be a meromorphic fibration, with $(Z,B_Z)$ being the orbifold base of a neat model of $f$ (see \cite[Definition 2.5]{Cam16}). Assume that $K_Z+B_Z$ is pseudoeffective, then $\kappa(K_Z+B_Z)=\nu(K_Z+B_Z)=0$. In particular, $(X,\Delta)$ is a special orbifold.
\end{theorem}
\begin{proof}
This result essentially proved by \cite[Theorem 10.3]{Cam16}, we give here a more analytic proof based on Proposition \ref{prop:positivity}. 
As in the proof of Proposition \ref{prop:Zhang}, let $\phi: (M,\Delta_M)\to(Z,B_Z)$ be a neat model of $f$ such that such that $M$ and $Z$ are projective manifolds (since $X$ is Moishezon), $B_Z=B_{\phi,\Delta_M}$ and we can write
\[
K_M+\Delta_M\sim_{\Q}\pi^\ast(K_X+\Delta)+F_+
\]
where $\pi:M\to X$ is the natural bimeromorphic morphism, $\Delta_M=\pi_\ast^{-1}\Delta+F_-$, $F_+$ and $F_-$ are effective $\pi$-exceptional divisors without comment components, and $\Delta_M$ and $F_+$ have SNC supports. Then we apply Proposition \ref{prop:positivity} to $\phi$ and $N=-(K_M+\Delta_M^v)+F_+$ and we obtain that $-\phi^\ast K_Z-\Delta_M^v-\Ram(\phi)+F_+$ is pseudoeffective. By \cite[Lemma 7.2 \& Remark 7.3]{Wan21}, we have that $-\phi^\ast(K_Z+B_Z)+E'$ is pseudoeffective for some exceptional divisor $E'$. On the other hand,  $K_Z+B_Z$ is pseudoeffective by our assumption, thus we can argue as in the proof of Proposition \ref{prop:Zhang} to obtain that for any semipositive closed $(1,1)$-current $T\in c_1(K_Z+B_Z)$, $T=[D_Z]$ for some effective divisor $D_Z$ such that $\phi^\ast D_Z$ is $\pi$-exceptional. This implies that the numerical dimension $\nu(K_Z+B_Z)=0$, and we conclude by the abundance conjecture established in this case \cite[Theorem 1.3]{Kaw13} (see also \cite[Theorem 0.1]{CKP12}). Moreover, if $f$ is a general type fibration, the additivity theorem \cite[Theorem 4.2]{Cam04} leads to $\dim Z=0$, which means that $(X,\Delta)$ is a special orbifold.
\end{proof}

\subsection{Generic nefness of the tangent sheaf}
As in \cite[Theorem 6.1]{Ou23}, \cite[Theorem 1.4]{MWWZ25} and \cite[Theorem 6.1, Proposition 6.5]{IJZ25}, the nefness of the anticanonical divisor implies the generic nefness of the tangent sheaf, here the generic nefness is defined as in \cite[Setup 6.2]{IJZ25}. Before stating the theorem, let us briefly recall the definitions in \cite{IJZ25}. Throughout this subsection, let $X$ be a compact normal variety in Fujiki class $\mathscr{C}$ of dimension $n$ and $\gamma\in H_{2n-2}(X,\R)$ be either the non-pluripolar product $\langle\alpha^n\rangle$ (\cite[\S 3.1]{IJZ25}) of a big class $\alpha\in H^{1,1}_\BottChern(X,\R)$ with vanishing property (\cite[Definition 3.23]{IJZ25}) or $\gamma=\alpha_1\cdots\alpha_{n-1}$ for nef classes $\alpha_i\in H^{1,1}_\BottChern(X,\R)$ (see \cite[Definition 2.2]{IJZ25}); note that $\gamma$ is a movable class by \cite[Remakr 3.11(5)]{IJZ25}. Let $\mathscr{E}$ be a torsion free sheaf on $X$, then the slope $\mu_\gamma(\mathscr{E})$ of $\mathscr{E}$ with respect to $\gamma$ is defined in \cite[Definition-Lemma 3.20]{IJZ25}, and in a standard way we can define $\mu_{\gamma,\min}(\mathscr{E})$ and $\mu_{\gamma,\max}(\mathscr{E})$, and the Harder-Narasimhan filtration exists for any $\mathscr{E}$ (with respect to $\gamma$), see \cite[Setup 6.2]{IJZ25}. We say that $\mathscr{E}$ is generically nef with respect to $\gamma$ if $\mu_{\gamma,\min}(\mathscr{E})\ge0$.

\begin{theorem}
Let $X$ be a compact variety in Fujiki class $\mathscr{C}$. Assume that there is an $\R$-divisor $\Delta\ge0$ on $X$ such that $(X,\Delta)$ is a lc pair and $-(K_X+\Delta)$ is nef. Then the tangent sheaf $T_X$ of $X$ is generically nef with respect to $\gamma$. 
\end{theorem}

\begin{proof}
The theorem results from \cite[Corollary 1.5]{CP25} and the positivity result Proposition \ref{prop:positivity} by a standard argument. We only sketch the proof here. If otherwise $\mu_{\gamma,\min}(T_X)<0$, by the standard argument as in \cite[Proof of Theorem 1.4]{Ou23}, we can find a saturated subsheaf $\mathscr{F}$ of $T_X$ that satisfies the condition (1) of \cite[Proposition 6.3]{IJZ25} (see also \cite[Proposition 2.2]{Ou23}) and $\mu_{\gamma}(T_X/\mathscr{F})<0$, so that $\mathscr{F}$ is an algebraically integrable foliation of which the closure of the leaves are rationally connected. Let $\psi:X\dashrightarrow Y$ be the meromorphic mapping defined by $\mathscr{F}$, and take a elimination of indeterminacies $\phi:M\to Y$ of $\psi$ with the natural modification $\pi:M\to X$ that we assume to be a log resolution of $(X,\Delta)$. Since $X$ is in Fujiki class $\mathscr{C}$, we can assume that $M$ and $Y$ are compact K{\"a}hler manifolds. As in the proof of Proposition \ref{prop:Zhang}, we can write
\[
K_M+\Delta_M\sim_\R\pi_\ast(K_X+\Delta)+F_+,
\]
and by Proposition \ref{prop:positivity} (with the same perturbation technique as in the proof of Proposition \ref{prop:Zhang}) we have that
\[
K_{M/Y}+\Delta_M-\pi^\ast(K_X+\Delta)-\Ram(\phi)+F_+
\]
is pseudoeffective and thus we have
\begin{equation}
\label{eq:psef}
(K_M+\Delta_M-\pi^\ast(K_X+\Delta)-\Ram(\phi)+F_+)\cdot\gamma_M\ge0,
\end{equation}
where $\gamma_M=\langle\pi^\ast\alpha^{n-1}\rangle$ if $\gamma=\langle\alpha^{n-1}\rangle$ ($\alpha$ is big with vanishing property) or $\gamma=\pi^\ast\alpha_1\cdot\pi^\ast\alpha_{n-1}$ ($\alpha_i$ is nef). On the other hand, let $\mathscr{F}_M$ be the foliation defined by $\phi$, then $\mathscr{F}_M$ coincides with $\pi^{[\ast]}\mathscr{F}$ outside $\Exc(\pi)$, so that $-K_{\mathscr{F}_M}\cdot\gamma_M=c_1(\mathscr{F})\cdot\gamma$ by \cite[Lemma 3.27]{IJZ25}. For the same reason we have $c_1(T_X)\cdot\gamma=(-\pi^\ast(K_X+\Delta)+\Delta_M)\cdot\gamma_M$. Moreover, since $K_{\mathscr{F}_M}\sim K_{M/Y}-\Ram(\phi)+E'$ for some (not necessarily effective) $\pi$-exceptional divisor $E'$, hence from \eqref{eq:psef} we have
\[
c_1(\mathscr{F})\cdot\gamma=-K_{\mathscr{F}_M}\cdot\gamma_M\le (-\pi^\ast(K_X+\Delta)+\Delta_M)\cdot\gamma_M=c_1(T_X)\cdot\gamma,
\]
which contradicts to our assumption that $\mu_{\gamma}(T_X/\mathscr{F})<0$. 
\end{proof}

Notice that the theorem is known in the following cases:
\begin{itemize}
\item $X$ is projective: \cite[Theorem 1.4]{Ou23}, 
\item $X$ is a compact K{\"a}hler manifold: \cite[Proposition 5.2]{MWWZ25} 
\item $X$ is a compact variety in Fujiki class $\mathscr{C}$ with klt singularities: \cite[Theorem 6.1 and Proposition 6.5]{IJZ25}. 
\end{itemize}

Moreover, as is pointed out by Prof. Masataka Iwai, the specialness of $X$ can be deduced directly from the generic nefness of the tangent sheaf, as the specialness can be characterized by the non-existence of Bogomolov sheaves (see \cite[Theorem 2.27]{Cam04}). 

\subsection{Final step of the proof of main theorems}
\label{ss:final_step}
%Finally, let us finish the proof of our main theorems:

\begin{proof}[Proof of {Corollary \ref{cor:main1}, Theorem \ref{thm:main1}, Theorem \ref{thm:main0}}]
%

%Suppose $X$ is singular and the other two cases are simpler.
By Lemma \ref{lemma:an} and Remark \ref{add}, we know that $\pi_1(X)$ is virtually nilpotent. Let $\alb_X: X \dashrightarrow \Alb_X$ be the Albanese map (of any desingularization) of $X$, then since $X$ has rational singularities by \cite[Lemma 8.1]{Kaw85} $\alb_X$ is an everywhere defined morphism, such that $\alb_{X'}=\phi\circ\alb_X$ for any desingularization $\phi: X'\to X$ of $X$. By \cite[Corollary 3.13]{MWWZ25} we see that $\alb_X$, and thus $\alb_{X'}$, are algebraic fiber spaces, then we can apply {\hyperref[prop:Alb_surj_pi1]{Proposition \ref*{prop:Alb_surj_pi1}}} (or {\hyperref[cor:Campana]{Corollary \ref*{cor:Campana}}}) to $X'$ to conclude that the torsion-free nilpotent completion of $G:=\pi_1(X')\simeq\pi_1(X)$ is Abelian, here the isomorphism $\pi_1(X')\simeq\pi_1(X)$ follows from \cite[Theorem 1.1]{Tak03}; but we have seen that $G$ itself is virtually nilpotent (Lemma \ref{lemma:an}), thus it is almost Abelian.  
\end{proof}

To finish this paper, we establish Conjecture \ref{conj:main1} in some special cases of higher dimensions. 
\begin{theorem}
\label{thm:main2} Let $X$ be a compact K\"ahler variety, then $\pi_1(X)$ is almost Abelian if one of the following conditions is satisfied:
\begin{enumerate}
%\item $(X,\Delta)$ is a three dimensional klt pair with $-(K_X+\Delta)$ being nef;
\item $-K_X$ is nef and there is a desingularization $\pi:Y\rightarrow X$ such that $-K_{Y/X}$ is relatively ample;
\item $-K_X$ is nef and there is a desingularization $\pi:Y\rightarrow X$ such that $-K_{Y/X}$ is relatively effective.
\end{enumerate}
\end{theorem}
\begin{proof} The point is that under the above two assumptions, for a fixed K\"ahler class $\theta$ and any $\epsilon>0$, we can construct a Ricci almost negative RCD($-\epsilon,n$) space \cite{GS,Sze25,FGS25}. Also, the Albanese map is surjective, so the previous argument applies.
\end{proof}
%%%%%%%%%%%%%%%%%%%%%%%%%%%%%%%%%%%%%%%%%%%%%%%%%%%%%
\medskip

\begin{appendices}

\section{Partial \texorpdfstring{$C^0$}{text} estimate: local case}
\label{appI}

In this appendix, we prove a local version of the partial $C^0$ estimate firstly proposed  by Tian in \cite{Tian90III}, see also \cite{DS14,Tian15,Zhang21,LS1,LS2} for important developments. We always assume that $(\hat X, d_\omega, \mu_\omega)$ is a compact $\RCD(-1, 2n)$ space and there is a Lipchitz surjective map $\iota:\hat X\rightarrow X$ as in the set-up Proposition \ref{Prop:antinefpair}, where $X$ is a klt K\"ahler variety and $\iota$ is a one to one map when restricted to the metric regular part of $\hat X$. 
\begin{proposition}\label{prop:inj}$\iota$ is injective.
\end{proposition}
To prove the proposition, we need some preliminary lemmas. The following $L^2$-estimate for $\dbar$-equation is a direct generalization for singular K\"ahler varieties. Fix a point $z\in X$ and let $\Omega:=\{\rho< 0\}$ be a small neighbourhood of $p$ contained in $X$, where $\rho$ is the restriction of a smooth PSH function on a smooth ambient space. Also let $\Omega^\circ$ be a smooth Zariksi open subset of $\Omega.$ Suppose $\omega:=\ddbar\varphi$ is a K\"ahler current on $\Omega$, where $\varphi\in L^\infty(\Omega)\cap C^\infty(\Omega^\circ)\cap PSH(\Omega)$. To be consistent with the notation of \cite{FGS25}, we let $L$ to  be the trivial line bundle on $\Omega$ and regard $e^{-\varphi}$ as a singular metric of $L$. 
\begin{lemma} \label{L241}For any $k\geq 2$ and smooth $L^k$-valued $(0,1)$-form $\tau$ satisfying 
\begin{enumerate}

\item $\dbar \tau =0$,

\smallskip

\item $Supp\, \tau \subset \subset \Omega^\circ$, 

\smallskip

\end{enumerate}
there exists an $L^k$-valued section $u$ such that %
\begin{equation}
\bar \partial u = \tau
 \end{equation}
on $\Omega$ and
$$\int_\Omega|u|^2_{h^k} (k\omega)^n \leq \frac{1}{4\pi} \int_\Omega |\tau|^2_{h^k, k\omega} (k\omega)^n. $$

\end{lemma}
\begin{proof}The lemma is proved by a trick of Demailly. We first pick a holomorphic function $f$ such that the hypersurface $D:=\{f=0\}$ on $\Omega$ covers $\mathcal{S}:=\Omega\setminus\Omega^\circ$. Note that $f$ is not necessarily irreducible. Then $\Omega \setminus D$ is a Stein manifold. There exists a smooth plurisubharmonic function $\psi$ on $\Omega\setminus D$ such that $\theta_D=\ddbar \psi$ is a complete K\"ahler metric on $\Omega\setminus D$. Then $$\omega_\epsilon = k\omega+ \epsilon \theta_D$$ is obviously also a complete K\"ahler metric. We now view $\tau$ as an $L^k\otimes K_{\Omega\setminus D}^{-1}$ valued $(n,1)$-form on $\Omega\setminus D$ and let 
$$h_\epsilon = h^ke^{- \epsilon \psi}\omega^n$$
be a hermitian metric on $L^k\otimes K_{\Omega\setminus D}^{-1}$ satisfying
$$\ric(h_\epsilon) = k\omega + \epsilon \theta_D + \ric(\omega)  \geq (k-1)\omega + \epsilon \theta_D  \geq \frac{1}{2} \omega_\epsilon. $$ 
Therefore by Demailly's $L^2$-estimate for $\dbar$-equation, there exist an $L^k\otimes K_{\Omega\setminus D}^{-1}$-valued $(n,0)$-form $u_\epsilon$ satisfying
$$\bar \partial u_\epsilon = \tau$$
on $\Omega\setminus D$ and
$$\int_{\Omega\backslash D} |u_\epsilon|^2_{h_\epsilon} \omega_\epsilon^n \leq \frac{1}{4\pi} \int_{\Omega\backslash D} |\tau|^2_{h_\epsilon, \omega_\epsilon} \omega_\epsilon^n.$$
We now let $\epsilon \rightarrow 0$. Since $u_{\epsilon_1} - u_{\epsilon_2}$ is holomorphic with uniform $L^2$-bounded on a fixed compact subset in $\Omega\setminus D$, it must be uniformly bounded and therefore $u_\epsilon$ converges to an $L^k$-valued  section $u$ satisfying 
$\dbar u = \tau$ on $\Omega\setminus D$ and $$ \int_\Omega |u|^2_{h^k} (k\omega)^n \leq \frac{1}{4\pi} \int_\Omega |\tau|^2_{h^k, k\omega} (k\omega)^n. $$
Since $\omega$ is smooth on $\Omega^\circ$, the $L^2$-bound for $u$ implies that $\dbar u = \tau$ holds on $\Omega^\circ$. Further note that $\tau$ is compactly supported on $\Omega^\circ$, so $\partial\bar u=0$ outside the support of $\tau$. By the $L^2$ bound of $u$, then $u$ extends to $\Omega$ as holomorphic function outside the support of $\tau$.   We now have completed the proof.
\end{proof}

%\begin{lemma}\cite[]{FGS25}\label{gradlocal}
%Suppose  $(\hat X, d_\omega, \mu_\omega)$ is a compact $\RCD(-1, 2n)$ space and there exists $c>0$ such that 
%
%$$\omega\geq c \theta_X .$$
%
%
%Let $\Omega$ be an open domain in $\hat X$ and $K$ be relatively compact subset of $\Omega$. Set $K^\circ=K\cap\Omega^\circ$, then  for any $\sigma \in H^0(\Omega, L^k)$, we have
%
%\begin{equation}\label{inibdsig}
%
%\sup_{K^\circ} |\sigma|^2_{h^k}   < \infty,
%
%\end{equation}
%
%\begin{equation}\label{inibdsid2}
 %
 % \sup_{K^\circ} |\nabla \sigma|^2_{h^k, k\omega}   < \infty.
%
%\end{equation}

%\end{lemma}
 %Direct calculations show that 
%
%\begin{equation}\label{lapsig41}
%
%\Delta_{k\omega}  |\sigma|^2_{h^k} = tr_{k\omega}\left( \left(\nabla \sigma \wedge  \overline{\nabla \sigma} \right) h^k \right) -  n|\sigma|^2_{h^k}  \geq - n|\sigma |^2_{h^k}.
%
%\end{equation}
%
%At any $p\in X^\circ$, by choosing normal coordinates at $p$ with respect to $\omega$, we can assume that $\nabla h(p) = 0$. Direct calculations  show that 
%
%\begin{equation}\label{lapsig42}
%
%\Delta_{k\omega} |\nabla \sigma|^2_{h^k, k\omega} \geq |\nabla^2 \sigma|^2_{h^k, k\omega} - (1+k^{-1})n |\nabla \sigma|^2_{h^k, k\omega} +n  |\sigma|^2_{h^k}
%
%\end{equation}
%
%by the Bochner formula. 

We denote $\Delta_\sharp$ and $\nabla_\sharp$ by the Laplace and gradient operators with respect to the rescaled metrics $h^k$ and $k\omega$ and similarly define the scaled norm $|| \cdot ||_{L^{\infty, \sharp}}$ and $|| \cdot ||_{L^{2, \sharp}}$ for $s\in H^0(\Omega, L^k)$ with respect to the hermitian metric $h^k$ and $k\omega$.

\begin{lemma}\cite{FGS25} \label{L242}  There exists  $\Lambda >0$ such that if  $s\in H^0(\Omega, L^k)$ for $k\geq 1$ and $B_1$ is a unit ball of $k\omega$ contained in $\iota^{-1}\Omega$, then 
\begin{equation}\label{l2form1}
\|s\|_{L^{\infty, \sharp}}(B_{1/2}) \leq \Lambda \|s\|_{L^{2, \sharp}}(B_1).
\end{equation}
\begin{equation}\label{l2form2}
 \|\nabla_\sharp s\|_{L^{\infty, \sharp}}(B_{1/2}) \leq \Lambda \|s\|_{L^{2, \sharp}}(B_1). 
\end{equation}
\end{lemma}

%\begin{proof} The estimates (\ref{lapsig41}) and (\ref{lapsig42}) imply that 
%
%$$\Delta_\sharp |\sigma|_{\sharp} \geq - n |\sigma|_{\sharp}, ~ \Delta_\sharp |\nabla_\sharp \sigma|_\sharp \geq - n |\nabla\sigma|_\sharp, $$
%
%where   $| \cdot |_\sharp$ are taken with respect to the rescaled metrics $h^k$ and $k\omega$.  Since $\cS(X)$ is an analytic subvariety of $X$ and has $0$-capacity, we can apply the same argument in \cite{S2} (c.f. Proposition 3.3 and Proposition 3.4) by combining Moser's iteration, the cut-off functions  for $|\sigma|_\sharp$ and $|\nabla_\sharp \sigma|$ in Lemma \ref{cutoff},  (\ref{inibdsig}) and (\ref{inibdsid2} in Lemma \ref{L242} to  obtain the $L^2$-estimates (\ref{l2form1}) and (\ref{l2form2}). 
%
%
%\end{proof}

Once Lemma \ref{l2form1} and Lemma \ref{l2form2} are proved, we can prove a version of local partial $C^0$ estimate.
\begin{lemma}\cite[Lemma 5.15]{FGS25}\label{particalc0local} Let $(V, o)$ be any tangent cone $(Y_\infty, p)$. Suppose the $\varepsilon$-singular set $V\setminus \mathcal{R}_\varepsilon(V)$ has $0$ capacity for some $\varepsilon>0$.  Let $q\in \iota^{-1}{\Omega}$ be a point different from $p$. Then for any given small $\xi$,  there exists $C>0, m>0$ and $\sigma \in H^0(\Omega, L^m)$ satisfying 
\begin{enumerate}

\item $\frac{1}{2}\leq |\sigma|_{h^m} \leq 1 $ on $B_1(p)\cap\Omega^\circ$, with respect to metric $m\omega$.

\smallskip

\item $|\sigma|_{h^m} \leq \xi$ on $B_1(q)\cap\Omega^\circ$, with respect to metric $m\omega$.

\end{enumerate}

\end{lemma}

Now we proceed to prove the main proposition of this appendix.
\begin{proof}[Proof of Proposition \ref{prop:inj}]
Suppose we have two different points $p,q\in\hat X$ such that $z=\iota(p)=\iota(q)$. Fix a neighbourhood $\Omega$ of $z$, such as in Lemma \ref{l2form1}, contained in $X$. Set $\Omega^\circ$ be the intersection of $\Omega$ and the metric regular part of $\hat X$. We also fix  two sequence of points $p_i\in\Omega^\circ$ and $q_i\in\Omega^\circ$ converging to $p,q$ separately in the topology of $\hat X$. Note that both $\{p_i\}$ and $\{q_i\}$ converge to $z$ in the topology of $X$. By lemma \ref{particalc0local}, there is an integer $m$  and two holomorphic functions $f_p$  such that $\frac{1}{2}\leq |f_p|_{h^m}\leq 1$ on $B_1(p)\cap\Omega^\circ$ and $|f_p|_{h^m} \leq \xi$ on $B_1(q)\cap\Omega^\circ$. Similarly we have $f_q$ such that $\frac{1}{2}\leq |f_q|_{h^m}\leq 1$ on $B_1(q)\cap\Omega^\circ$ and $|f_q|_{h^m} \leq \xi$ on $B_1(p)\cap\Omega^\circ$. We first assume that $f_p(z)=0$, then $|f_p|(p_i)\rightarrow 0,$ as $i\rightarrow \infty$, this contradicts with $|f_p|_{h^m}\geq\frac{1}{2}$. Now if we set $g:=\frac{f_q}{f_p}$, then $g$ is a well defined holomorphic function near $z$. We also have $|g|(p_i)\leq 2\xi$ and $g$ and $|g|(q_i)\geq \frac{1}{2\xi}$. By choosing $\xi$ sufficiently small, this contradicts with $g$ is continuous at $z$. The proposition is proved. \end{proof}

\medskip

\section{A segment inequality for K{\"a}hler spaces}

%Our goal is to show that the morphism \eqref{mor} is surjective for $\epsilon$ sufficiently small. Of course, if one could show that the diameter of $\omega_\epsilon$ is uniformly bounded, then the surjectivity follows easily. However, the uniform diameter upper bound do not follow easily, we will need the following nice observation of \cite{DPS93}. Here the singularities do not cause essential trouble.

Throughout this section, we will assume that $(X, \omega)$ is a compact K{\"a}hler variety with a smooth K{\"a}hler form $\omega$. We also fix a subvariety $\cS$ of $X$ such that $X\setminus \mathcal S$ is connected. For any  $\phi \in \PSH(X, \omega)$, we let $\omega_\phi = \omega+ \ddbar \phi$. The following Proposition is  essentially due to \cite{DPS93}.

\begin{proposition} \label{DPS} For any compacts subsets $U_1,U_2 \subset X\setminus \cS$, there exists $C = C(U_1, U_2)>0$ such that for any $\phi \in \PSH(X, \omega)\cap L^\infty(X) \cap C^2(X\setminus \cS)$  and any $\delta \in (0, 1)$, there exist  closed subsets $U_{1,\delta}\subset U_1$ and $U_{2,\delta}\subset U_2$  satisfying the following.
\begin{enumerate}

\item  $\Vol_\omega(U_i \setminus U_{i,\delta})<\delta$, $i=1, 2$; 

\smallskip

\item any two points $x_1\in U_{1,\delta}, x_2\in U_{2,\delta} $ can be joined by a path $\gamma \subset X\setminus \cS$  with 
$$\cL_{\omega_\phi}(\gamma) \leq C\delta^{-1/2}, $$
where $\cL_{\omega_\phi}(\gamma)$ is the arc length of $\gamma$ with respect to $\omega_\phi$.
\end{enumerate}

\end{proposition}

\begin{proof} We sketch the proof for readers convenience. 
The key is the following identity, since $\omega_\phi$ is in the same class as $[\omega]$, we have
$$\int_X \omega_\phi\wedge \omega^{n-1}=\int_X \omega^n.$$
First suppose that $U_1=U_2=K$ where $K$
is some compact convex set contained in some coordinate chart of $\tilde X$. We simply join $x_1\in K$ and $x_2\in K$ by a line segment $[x_1,x_2]$ in $K$ and compute the average of $\omega_\phi$-length with respect the Lebesgue measure of $K$. By Fubini and the Cauchy-Schwarz inequality we get
\begin{align*}\int_{K\times K}\Bigg(\int_0^1\sqrt{\omega_\phi(tx_2+(1-t)x_1)(x_2-x_1)}dt\Bigg)^2dx_1dx_2 \\
\leq|x_2-x_1|\int^1_0 dt \int_{K\times K}\omega_\phi(tx_2+(1-t)x_1)dx_1dx_2\\
\leq 2^{2n}\cdot diam_\omega(K) \cdot \Vol_\omega(K)\cdot \|\omega_\phi\|_{L^1(K)}\leq C_1
\end{align*}
where $C_1$ is independent of $\phi$ and the last inequality is obtained by integrating first with respect to $y=(1-t)x_1$ when $t\leq \frac{1}{2}$, resp to $y=tx_2$ when $t\geq\frac{1}{2}$ (Note that $dx_j\leq 2^{2n}$ in both cases).\\
It follows that the set $S$ of pairs $(x_1, x_2)\in K\times K$ such that $\omega_\phi$-length of $[x_1,x_2]$ exceeds $(C_1)\delta^{-1/2}$ will have measure less than $\delta$ in $K\times K$. By Fubini Theorem, the set $Q$ of points $x_1\in K$ such that the slice $S(x_1):=\{x_2:(x_1,x_2)\in S\}$  has volume $\Vol(S(x_1))\geq \frac{1}{2}\Vol(K)$  itself has measure $\Vol(Q)\leq 2\delta/\Vol(K)$. Now for $x_1,x_2\in K\setminus Q$, we have $\Vol(S(x_j))\leq \frac{1}{2}\Vol(K)$, therefore 
$$\big(K\setminus S(x_1)\big)\cap \big(K\setminus S(x_2)\big)\neq\emptyset.$$
If $y$ is a point in this above set, we have $(x_1,y)\notin S$ and $(x_1,y)\notin S$, so 
$$\textnormal{length}([x_1,y]\cup[y,x_2])\leq 2C_1\delta^{-1/2}$$
Now we set $U_{j,\delta,\phi}=K\setminus Q$ and we have
$$\Vol(U_j\setminus U_{j,\delta,\phi})\leq \Vol(Q) \leq 2C_1\delta/\Vol(K).$$
By modifying $\delta$, we get the desired $\delta$ volume bound of $\Vol(U_j\setminus U_{j,\delta,\phi})$.

If $U_1$ and $U_2$ are two different compact convex sets in $X\setminus\mathcal S$, we just need to choose a finite sequence of compact convex sets $V_i,i=1,2...N$ contained in the smooth part of $\tilde X$ such that $U_1=V_1$, $U_2=V_N$ and $V_i\cap V_{i+1}$ contains an open set (This is possible because the complex analytic regular part of $\tilde X$ is connected). Note that, the choice of $V_i$ is independent of the K\"ahler metrics. Applying the previous argument to each $V_i$ with $V_{i+1,\delta,\phi}\cap V_{i,\delta,\phi}\neq\emptyset$ will finish the proof. The case when $U_1,U_2$ are not convex is obtained by covering these sets with finitely many compact convex coordinate charts.
\end{proof}

\medskip

\section{Surjectivity of the Albanese map and the almost abelianity of the fundamental group}
\label{app:Campana}
F. Campana proves that the virtual nilpotence and the almost abelianity are the same for the fundamental group of compact K{\"a}hler manifolds of which the Albanese map is surjective.
\begin{proposition}[{\cite[Corollaire 3.1]{Cam95}}]
\label{prop:Alb_surj_pi1}
Let $Y$ be a compact K\"ahler manifold such that the Albanese map of $Y$ is surjective, then the nilpotent completion of $\pi_1(Y)$ is Abelian.
\end{proposition}

For the convenience of the readers, we will give the proof of the proposition by closely following the argument in \cite{Cam95}. First recall some notions from group theory: let $G$ be a group, and let 
\[
G=G_0\supseteq G_1\supseteq \cdots\supseteq G_i \supseteq G_{i+1}\supseteq\cdots
\]
be the central series of $G$, i.e. we inductively define $G_{i+1}:=[G,G_i]$. And we put $G_{\infty}:=\bigcap_{i\geq 1}G_i$. For any $n\in\Z_{>0}\cup\{\infty\}$, we set $G'_n:=\sqrt{G_n}$, here for any subset $S\subseteq G$ we define 
\[
\sqrt S:=\left\{g\in G\;\big|\;g^n\in S\right\}.
\]
The \emph{$($torsion free$)$ nilpotent completion} of $G$ is defined to be $G^\nilp:=G/G'_\infty$\,. The proof of {\hyperref[prop:Alb_surj_pi1]{Theorem \ref*{prop:Alb_surj_pi1}}} is based on the following characterization due to J. Stallings:

\begin{theorem}[{\cite[Theorem 3.4]{Sta65}}]
\label{thm:Stallings}
Let $f:G\to H$ be a group morphism. Assume that $H_1(f): H_1(G,\Q)\to H_1(H,\Q)$ is an isomorphism and $H_2(f): H_2(G,\Q)\to H_2(H,\Q)$ is surjective. Then 
\begin{itemize}
\item[\rm(a)] the induced morphism $f_k: G/G'_k\to H/H'_k$ is injective for $k\in\Z_{>0}\cup\{\infty\}$ and is of finite index for $k\in\Z_{>0}$\,$;$
\item[\rm(b)] if $f$ is surjective, then $f_k$ is an isomorphism for each $k\in\Z_{>0}\cup\{\infty\}$.
\end{itemize}
\end{theorem}
\begin{proof}[Sketch of the proof]
The theorem can be proved by induction combining with the following exact sequence established in \cite[Theorem 2.1]{Sta65}: 
\[
H_2(G) \to H_2(G/N) \to N/[G,N] \to H_1(G) \to H_1(G/N) 
\]
where $N\subseteq G$ is normal subgroup, and this sequence is functorial with respect to $f:(G,N)\to (G',N')$ where $f: G\to G'$ is a group morphism such that $f(N)\subseteq N'$. 
\end{proof}

Now let us turn to the proof of {\hyperref[prop:Alb_surj_pi1]{Proposition \ref*{prop:Alb_surj_pi1}}}. Indeed, we will prove a stronger result; in order to state it, let us first set up some notations: let $Y$ be a compact K\"ahler manifold and $\alb_Y: Y\to\Alb_Y$ be the Albanese map of $Y$, let $W$ be a desingularization of the image of $\alb_Y$. Take an elimination of indeterminacies $Y'\to W$ of $\alpha: Y\dashrightarrow W$ with the natural (birational) morphism $\mu:Y'\to Y$. Since $\mu$ induce an isomorphism on fundamental groups, we may assume that $\alpha$ is an
%\textcolor{red}{Typo? NO, here `a(n)' represents for `a morphism' and `an everywhere defined morphism'} 
everywhere defined morphism. 
%and since $\alb_{Y'}=\alb_Y\circ\mu_Y$, thus up to replacing $Y$ and $W$ by $Y'$ and $W'$ respectively, we may assume that $W$ is smooth (thus a compact K\"ahler manifold). 
We set $G:=\pi_1(Y)$ and $H:=\pi_1(W)$, then $\alpha$ induces a group morphism $\alpha_\ast:G\to H$. In the sequel we will prove the following:
\begin{theorem}[{\cite[Th\'eor\`eme 2.1]{Cam95}}]
\label{thm:Campana}
The induced morphism $\alpha_{\ast,n}:G/G'_n\to H/H'_n$ is injective for $n\in\Z_{>0}\cup\{\infty\}$ and is of finite index for $n\in\Z_{>0}$\,. Moreover, if $\alpha_\ast$ is surjective, then $\alpha_{\ast,n}$ is an isomorphism for each $n\in\Z_{>0}\cup\{\infty\}$.
\end{theorem}

\begin{proof}
In view of {\hyperref[thm:Stallings]{Theorem \ref*{thm:Stallings}}}, we just need to check that the following two conditions hold:
\begin{itemize}
\item[(i)] $H_1(\alpha):H_1(G,\Q)\to H_1(H,\Q)$ is an isomorphism; 
\item[(ii)] $H_2(\alpha): H_2(G,\Q)\to H_2(H,\Q)$ is surjective.
\end{itemize} 
To this end, we need to associate the cohomology of a complex variety (or more generally, a cell complex) to the (group) homology of its fundamental group. This can be done by the following construction: 

Let $K$ be a connected cell complex, by adding cells of dimension $\geq 3$ to $K$ we can construct an aspherical complex $K'$ so that $K$ is embedded into $K'$. From the construction we see that the higher homotopy groups of $K'$ are all vanishing, hence $K'=K(\pi_1(K),1)$ the Eilenberg-Maclane space of $\pi_1(K)$, in particular $H_n(K')\simeq H_n(\pi_1(K))$ for any $n$. Since $K'\backslash K$ does not contain any cell of dimension $\leq 2$, it follows that the natural morphism $H_1(K)\to H_1(K')$ is an isomorphism and $H_2(K)\to H_2(K')$ is surjective by looking at the homology sequence of the pair $(K',K)$:
\[
H_2(K)\to H_2(K')\to \underbrace{H_2(K',K)}_{=0}\to H_1(K)\to H_1(K')\to \underbrace{H_1(K',K)}_{=0}\to 0.
\]
Hence $H_1(K)\simeq H_1(\pi_1(K))$ and $H_2(K)\twoheadrightarrow H_2(\pi_1(K))$. 

Now (i) follows from the construction of the Albanese map: the morphism of homology group $H_1(\alb_Y): H_1(Y,\Z)/\text{tors}\xrightarrow{H_1(\alpha)} H_1(W,\Z)\to H_1(\Alb_Y,\Z)$ is an isomorphism. In order to establish (ii), we first notice the following lemma essentially due to J. Stallings (see \cite[Theorem 5.1]{Sta65}):
\begin{lemma}[{\cite[Lemme 2.3]{Cam95}}]
\label{lemma:Stallings}
Let $\phi: K\to L$ be a morphism of finite connected cell complexes. Assume that $H^2(\phi): H^2(L,\Q)\to H^2(K,\Q)$ is injective, then $H_2(\phi_\ast):H_2(\pi_1(K),\Q)\to H_2(\pi_1(L),\Q)$ is surjective.
\end{lemma}
\begin{proof}
Note that we have the following commutative diagram:
\begin{center}
\begin{tikzpicture}[scale=3.0]
\node (A) at (0,0) {$H_2(L,\Q)$};
\node (A1) at (1,0) {$H_2(\pi_1(L),\Q)$,};
\node (B) at (0,1) {$H_2(K,\Q)$};
\node (B1) at (1,1) {$H_2(\pi_1(K),\Q)$};
\path[->,font=\scriptsize,>=angle 90]
(B) edge node[left]{$H_2(\phi)$} (A)
(B1) edge node[right]{$H_2(\phi_\ast)$} (A1);
\path[->>]
(A) edge (A1)
(B) edge (B1);
\end{tikzpicture}  
\end{center}
where the horizontal arrows are constructed above and they are surjective. By our assumption, the left arrow $H_2(\phi): H_2(K,\Q)\to H_2(L,\Q)$ is surjective, thus by 5-lemma, so is the right arrow $H_2(\phi_\ast)$.   
\end{proof}

From the lemma above we see that (ii) follows immedately from the injectivity of $H^2(\alpha): H^2(W,\Q)\to H^2(Y,\Q)$, which results from the more general {\hyperref[lemma:Kahler-cohomology]{Lemma \ref*{lemma:Kahler-cohomology}}} below. Thus we complete the proof of {\hyperref[thm:Campana]{Theorem \ref*{thm:Campana}}}.
\end{proof}

\begin{lemma}[cf.~{\cite[Lemme 2.4]{Cam95}}]
\label{lemma:Kahler-cohomology}
Let $\alpha: X\to Y$ be a surjective proper morphism between (not necessarily compact) connected complex manifolds, and assume that $X$ is K\"ahler. Then $H^i(\alpha): H^i(Y,\R)\to H^i(X,\R)$ is injective for every $i\in\Z_{\geq 0}$.
\end{lemma}
\begin{proof}
The idea of this result has already appeared in A. Fujiki's work \cite[Proof of Proposition 1.6]{Fuj78}. Let us remark that the compactness of $X$ (and of $Y$) is not necessary as long as we assume the morphism $\alpha$ is proper. Take a K\"ahler form $\omega$ on $X$, and set $r:=\dim X-\dim Y$. Let $Y_0$ be the Zariski open of $Y$ over which $\alpha$ is smooth, then we claim that the integral $c(y):=\int_{X_y} \omega|_{X_y}^r$ is independent of $y\in Y_0$. Indeed, it suffices to show $c(\cdot)$ is locally constant, hence by Ehresmann's theorem (see \cite[Proposition 9.3]{Voi02}) we can assume that $X_0:=\alpha^{-1}(Y_0)$ is diffeomorphic to a product $Y_0\times F$ where $F$ is a general fibre of $\alpha$. Then we can write $c(y)=\int_F\tau^\ast\omega|^r_{\{y\}\times F}$ where $\tau:Y_0\times F\xrightarrow{\simeq} X_0$ is the diffeomorphism given by Ehresmann's theorem, hence by $d$-closedness of $\omega$ we have
\[
dc(y) = \int_F d(\tau^\ast\omega|^r_{\{y\}\times F})=0.
\]
The claim is thus proven, and we simply denote this constant by $c$. Then for any real $2\dim Y$-form $\phi$ with compact support on $Y$, we have 
\[
\int_X\alpha^\ast\phi\wedge\omega^r=\int_{X_0}\alpha^\ast\phi\wedge\omega^r=c\cdot\int_{Y_0}\phi=c\cdot\int_Y\phi.
\]
Now take $[u]\in\ker(H^i(\alpha))$, then $\alpha^\ast u$ is $d$-exact, hence for any $[v]\in H^{2\dim Y-i}_c(Y,\R)$ we have
\begin{align*}
0 &=\int_X\alpha^\ast(u\wedge v)\wedge\omega^r=c\int_Y u\wedge v,\\
 & \uparrow  \\
 & \text{\tiny by Stokes} 
\end{align*}
thus $[u]=0$ by Poincar\'e duality.
\end{proof}

As a corollary of {\hyperref[thm:Campana]{Theorem \ref*{thm:Campana}}}, we have the following result, which is nothing but a refinement of {\hyperref[prop:Alb_surj_pi1]{Proposition \ref*{prop:Alb_surj_pi1}}}.
\begin{corollary}[{\cite[Corollaire 3.1]{Cam95}}]
\label{cor:Campana}
Let $Y$ be a compact K\"ahler manifold and set $G:=\pi_1(Y)$. Assume that the Albanese map of $Y$ is surjective. Then 
\begin{itemize}
\item[\rm(a)] $f:=\alb_{Y\ast}:G\to\pi_1(\Alb_Y)$ is surjective.
\item[\rm(b)] $G'_1=G'_n=G'_\infty$ for any $n\geq1$, and $G^{\nilp}\simeq\pi_1(\Alb_Y)$.
\end{itemize}
\end{corollary}
\begin{proof}
By virtue of {\hyperref[thm:Campana]{Theorem \ref*{thm:Campana}}}, it suffices to establish (a). Since $\alb_Y$ is surjective, from \cite[Proposition 1.3]{Cam91} we know that $f$ is of finite index $d$. Let $h: T\to\Alb_Y$ be an isogeny of degree $d$, such that $h_\ast(\pi_1(T))=f(G)$, then by the lifting lemma we get a factorization $\beta: Y\to T$ such that $\alb_Y=h\circ\beta$; from the universal property of the Albanese map, $h$ is a fortiori an isomorphism, and $f=\alb_{Y\ast}$ is thus surjective.   
\end{proof}

\section{Non-nef loci and asymptotic multiplier ideals}
\label{app:nonnef}
In this appendix, we will show some results related to the non-nef loci of pseudoeffective classes. First recall the definition of the non-nef loci of pseudoeffective classes:
\begin{definition}
Let $X$ be a compact K{\"a}hler manifold and let $\alpha\in H^{1,1}(X,\R)$ be a pseudoeffective class. Pick a K{\"a}hler form $\omega$ on $X$. As in Definition \ref{def:asym_klt}, for every $\epsilon>0$, let $\alpha[-\epsilon\omega]$ denotes the set of currents $T\in\alpha$ such that $T\ge-\epsilon\omega$ and let $T_{\min,\epsilon}$ be a current with minimal singularities in $\alpha[-\epsilon\omega]$ (see \cite[\S 2.8]{Bou04}) and write $T_{\min,\epsilon}=\theta_\epsilon+dd^c\varphi_\epsilon$ where $\theta_\epsilon\in\alpha$ is a smooth $(1,1)$-form and $\varphi_\epsilon$ is quasi-psh. Then the asymptotic multiplier ideal of $\alpha$ is defined to be (see Remark \ref{rmk:asym_klt}(a))
\[
\mathscr{J}_+(\alpha):=\bigcap_{\epsilon>0}\mathscr{J}(\varphi_\epsilon).
\]
The minimal multiplicity of $\alpha$ at $x\in X$ is defined to be (see \cite[Definition 3.1]{Bou04})
\[
\nu(\alpha,x):=\sup_{\epsilon>0}\nu(T_{\min,\epsilon},x).
\]
And the non-nef locus of $\alpha$ is defined to be (see \cite[Definition 3.3]{Bou04})
\[
\NNef(\alpha)=\left\{x\in X\,\big|\, \nu(\alpha,x)>0\right\}.
\]
It is clear that these definitions depend neither on the current with minimal singularities chosen nor on the K{\"a}hler form $\omega$. 
\end{definition}

First let us prove the equality \eqref{eq:non-nef-locus}, which can be regarded as the analytic version of the fundamental relation established by the famous work of Ein-Lazarsfeld-Musta{\cb{t}}{\u{a}}-Nakamaye-Popa \cite[Corollary 2.10]{ELMNP06}.

\begin{lemma}
\label{lemma:nnef=I+}
Let $X$ be a compact K{\"a}hler manifold and $\alpha\in H^{1,1}(X,\R)$ a pseudoeffective class. Then 
\[
\NNef(\alpha)=\bigcup_{m\in\Z_{>0}}V(\mathscr{J}_+(m\alpha))_{\red}.
\]
Here, for any ideal sheaf $\mathscr{I}\subseteq\mathscr{O}_X$, $V(\mathscr{I})$ denotes the complex subspace defined by $\mathscr{I}$.
\end{lemma}
\begin{proof}
Pick a K{\"a}hler form $\omega$ on $X$. If $\nu(\alpha,x)=0$, then for every $\epsilon>0$, $\nu(T_{\min,\epsilon},x)=0$, in particular, we have $\nu(\varphi_\epsilon, x)<1$ which implies that $\mathscr{J}(\varphi_\epsilon)_x=\mathscr{O}_{X,x}$ and hence we have $\mathscr{J}_+(\alpha)_x=\mathscr{O}_{X,x}$ by Skoda's theorem \cite[Lemma (5.6.a), p.40]{Dem10}. Moreover, by homogeneity $\nu(m\alpha,x)=0$ for every $m\in\Z_{>0}$ and the same argument shows that $\mathscr{J}_+(m\alpha)_x=\mathscr{O}_{X,x}$. This establishes the `$\supseteq$' part of the equality. 

If $\nu(\alpha,x)>0$, again by homogeneity \cite[Proposition 3.5]{Bou04} we have $\nu(m\alpha,x)> n$ for $m\in\Z_{>0}$ sufficiently large. Hence, for every $\epsilon>0$ sufficiently small, we have 
\[
\nu(\varphi_{m,\epsilon},x)=\nu(T_{m,\min,\epsilon},x)>n, 
\]
where $T_{m,\min,\epsilon}$ is a current with minimal singularities in $m\alpha[-\epsilon\omega]$ and $\varphi_{m,\epsilon}$ is a local potential of $T_{m,\min,\epsilon}$. Again by Skoda's theorem \cite[Lemma (5.6.b), p.40]{Dem10} $\mathscr{J}_+(m\alpha)_x\subsetneqq\mathscr{O}_{X,x}$, and thus the `$\subseteq$' part of the equality is established. 
\end{proof}

If $f:X\to Y$ is a surjective morphism between projective varieties, and $D$ a pseudoeffective divisor on $Y$ then it is easy to check (by definition) that $\B_-(f^\ast D)\subseteq f^{-1}\B(D)$. The analytic analogue of this fact holds:

\begin{lemma}
\label{lemma:pullback-nnef}
Let $f:X\to Y$ to be surjective morphism between compact K{\"a}hler manifolds and let $\alpha\in H^{1,1}(Y,\R)$ be a pseudoeffective class. Then $\NNef(f^\ast\alpha)\subseteq f^{-1}\NNef(\alpha)$. 
\end{lemma}
In order to prove the lemma, a crucial observation is the following simple fact:
\begin{lemma}
\label{lemma:pullback-Lelong-nb-0}
Let $\varphi$ be a quasi-psh function on $Y$ and let $f:X\to Y$ be a morphism between compact K{\"a}hler manifolds such that $f(X)$ is not contained in $\varphi^{-1}(-\infty)$. If $\nu(\varphi,y)=0$ for some $y\in f(X)$, then $\nu(\varphi\circ f, x)=0$ for every $x\in f^{-1}(y)$.
\end{lemma}

This can be checked directly by definition of Lelong numbers \cite[Definition (5.5), p.40]{Dem10}, thus we omit the proof.

\begin{proof}[Proof of {Lemma \ref{lemma:pullback-nnef}}]
Pick K{\"a}hler forms $\omega_X$ and $\omega_Y$ on $X$ and $Y$ respectively and we may assume that $\omega_X\ge f^\ast\omega_Y$. If $\nu(\alpha,y)=0$, then for every $\epsilon>0$, $\nu(T_{\min,\epsilon},y)=0$ and this implies that $\nu(f^\ast T_{\min,\epsilon},x)=0$ for every $x\in f^{-1}(y)$ by Lemma \ref{lemma:pullback-Lelong-nb-0}. Here, the pullback of currents is taken in the sense of \cite[\S 2.2.3]{Bou04}. Moreover, we have $f^\ast T_{\min,\epsilon}\ge-\epsilon f^\ast\omega_Y\ge-\epsilon\omega_X$, i.e., $f^\ast T_{\min,\epsilon}\in f^\ast\alpha[-\epsilon\omega_X]$. Hence $\nu(f^\ast\alpha,x)=0$ for every $x\in f^{-1}(y)$. Thus the lemma is proved.
\end{proof}

\end{appendices}

    %%%%%%%%%%%%%%%%%%%%%%%%%%%%%%%%%%%%%%%%%%%%

\bigskip
\bigskip

\bibliographystyle{plain}
\bibliography{FG}

\end{document}